\documentclass[11pt]{article}

\usepackage{amsmath}
\usepackage{amssymb}
\usepackage{amsthm}
\usepackage{tikz}
\usepackage{mathtools}
\usepackage{arydshln}

\usepackage{multirow}

\newtheorem{theorem}{Theorem}[section]
\newtheorem{lemma}{Lemma}[section]
\newtheorem{proposition}{Proposition}[section]

\newtheorem{corollary}{Corollary}[section]
\theoremstyle{definition}

\newtheorem{problem}{Problem}[section]

\newtheorem{definition}{Definition}[section]

\theoremstyle{plain}

\newcommand{\varA}{\mathcal{A}}

\newcommand{\varP}{\mathcal{P}}
\newcommand{\varN}{\mathcal{N}}
\newcommand{\varL}{\mathcal{L}}
\newcommand{\varR}{\mathcal{R}}

\newcommand{\tabP}{\mathmakebox[\widthof{$\mathcal{P}$}]{\mathcal{P}}}
\newcommand{\tabN}{\mathmakebox[\widthof{$\mathcal{P}$}]{\mathcal{N}}}
\newcommand{\tabL}{\mathmakebox[\widthof{$\mathcal{P}$}]{\mathcal{L}}}
\newcommand{\tabR}{\mathmakebox[\widthof{$\mathcal{P}$}]{\mathcal{R}}}

\newcommand{\rmL}{\mathbf{L}}
\newcommand{\rmR}{\mathbf{R}}
\newcommand{\rmP}{\mathbf{P}}
\newcommand{\rmN}{\mathbf{N}}

\newcommand{\x}{\mathbf{x}}
\newcommand{\e}{\mathbf{o}}
\newcommand{\m}{\mathbf{T}}

\newcommand{\varG}{\mathcal{G}}
\newcommand{\varH}{\mathcal{H}}

\newcommand{\varU}{\mathcal{U}}
\newcommand{\varV}{\mathcal{V}}
\newcommand{\varW}{\mathcal{W}}
\newcommand{\varT}{\mathcal{T}}

\newcommand{\nxt}{\operatorname{next}}
\newcommand{\prv}{\operatorname{prev}}
\newcommand{\mex}{\operatorname{mex}}

\newcommand{\col}{\textbf{\;\!:\!\;}}

\begin{document}

\title{A Finite Automaton Approach to Combinatorial Games}

\author{Kai Liang\thanks{School of Mathematics and Statistics, Xidian University, 266 Xinglong Section, Xifeng Road, 710126 Xi'an, Shaanxi, China \\ 24071213197@stu.xidian.edu.cn}}
\date{\today}
\maketitle

\begin{abstract}
    This study applies finite automata to the automatic solving of a variety of combinatorial games.
    For games whose positions and moves can be represented as regular languages and their operations,
    we design a two-stage automatic solving algorithm: first, construct a candidate finite automaton to determine the $\varP$- and $\varN$-positions, and then perform rigorous formal verification on this automaton; once verified, a complete solution of the game is obtained.

    For partizan octal games, we introduce a generalized misère quotient, overcoming the limitation that traditional theory applies only to impartial games.
    Using the above algorithm, we successfully solve the majority of two-digit partizan octal games, and based on these results, we propose a partizan version of Guy's conjecture.
    We also successfully solve a considerable number of partizan octal games under misère play, and give a conjecture on the structure of those games exhibiting ``algebraic periodicity'' among them.

    For Kotzig's nim, we resolve the most important related conjecture: we prove that the outcomes and SG values are periodic under both normal and misère play (including their partizan versions).
    Our algorithm successfully solves several small-scale cases, including misère play and partizan versions.

    This study pioneers a new theoretical tool and algorithmic paradigm for the automatic solving of combinatorial games, and has broad prospects for further extension and application in the field of combinatorial game theory.
\end{abstract}

\section{Introduction}

In combinatorial game theory, the use of finite automata for theorem proving has seen significant progress in recent years.

Wythoff's game, owing to the deep connection between its $\varP$-positions and the Fibonacci sequence, has become the most typical application scenario of this approach.
Fraenkel proved that the $\varP$-positions of Wythoff's game can be directly characterised by Fibonacci representations~\cite{1982Wyth}.
Building on this, Landman used finite automata to prove the arithmetic periodicity of the SG values of Wythoff's game~\cite{2002Wyth}.
Duch\^ene et al. systematically studied the problems of adding and restricting moves while preserving the $\varP$-positions of Wythoff's game; their proofs rely on Fibonacci words, automatic sequences, and related numeration systems,
and they gave methods for generating $\varP$-positions and encodings of addable moves via two-dimensional morphisms~\cite{2010Wyth}.
In recent years, the automatic proof tool Walnut, based on the B\"uchi--Bruyère theorem, has been introduced into this field, providing a powerful supplement to traditional methods.
Mignoty, Renard, Rigo, and Whiteland were the first to systematically apply Walnut to combinatorial game theory; they not only automatically reproved several known results on variants of Wythoff's game,
but also proved a conjecture of Duch\^ene et al. concerning redundant moves and proposed several new conjectures~\cite{2025Auto,2025Wyth}.

Building on previous work that mainly focused on Wythoff's game as a specific impartial game,
this study proposes a systematic algorithm that extends the finite-automaton method to a broader class of games, achieving notable breakthroughs at both the theoretical and computational levels.

First, for partizan octal games and Kotzig's nim, we encode positions and moves as strings and operations on strings, respectively.
Such a representation is not only suitable for these two classes of games but also enjoys good extensibility, and can be generalized to more combinatorial games whose positions are based on multiple heaps or one-dimensional boards, thus laying the foundation for subsequent automaton modelling.

Based on the above representation, we design a two-stage automatic solving algorithm:
In the first stage, the algorithm attempts to construct a finite automaton that can determine the positions with each outcome or SG value;
in the second stage, the algorithm automatically performs rigorous formal verification of the constructed automaton using automaton operations.
If verification succeeds, the automaton constitutes a complete solution of the game: for any position, by scanning its encoding from left to right, one can determine its outcome or SG value in $O(n)$ time and $O(1)$ space, where $n$ is the length of the position encoding.

For partizan octal games, we overcome the difficulty that the traditional misère quotient theory applies only to impartial games, and introduce the notion of a generalized misère quotient based on the action of a semigroup on a set.
This construction can be fully generated automatically from the results of the above automatic solving algorithm, thus for the first time filling the theoretical gap for non-dicotic partizan octal games under misère play.
Through analysis of the computational results, we propose a strengthened version of Guy's conjecture.
Meanwhile, for a class of misère octal games exhibiting ``algebraic periodicity'' that the algorithm cannot yet handle, we also give a conjecture on their structure.

For Kotzig's nim, by using finite-automaton methods to study their linear variants, we prove the most crucial related conjecture:
under both normal and misère play, the outcomes and SG values of all impartial and partizan Kotzig's nim are (ultimately) periodic.
Our algorithm successfully solves several new small-scale Kotzig's nim, including for the first time cases under misère play and partizan versions, filling a gap in previous research.

\section{Definitions and Notation}
\subsection{Combinatorial Games}

The games studied in this paper are \textit{combinatorial games} satisfying the following conditions:

(1)~\textit{Two players alternate}: the two players take turns moving.

(2)~\textit{Perfect information}: all rules, the current state, and all possible moves are public and transparent to both players.

(3)~\textit{No randomness}: the game involves no randomness such as dice rolling or card shuffling; the outcome is entirely determined by the players' decisions.

Sometimes we also require combinatorial games to satisfy the following conditions:

(1)~\textit{Normal play}: the player who makes the last move wins; equivalently, the player who cannot move loses.
The opposite of this is \textit{misère play}, where the player who cannot move wins.

(2)~\textit{Impartiality}: the two players have exactly the same legal moves available in any given position (i.e., determined solely by the position, not by the identity of the player).
Games satisfying impartiality are called \textit{impartial games}; otherwise they are called \textit{partizan games}.

In this paper, we generally denote by $\varU$ the set of all possible positions of a given game.
For a position $u\in\varU$, we denote by $\nxt(u)$ the set of all its (reachable in one move) options, and
\[
    \prv(u) \coloneq \{u' \mid u'\in\varU,\ u\in\nxt(u')\}
\]
the set of all its predecessors, i.e., the inverse of $\nxt(\cdot)$.

Furthermore, the \textit{outcome} of a position $u$ is defined as follows: if the current player has a winning strategy, then $u$ is an $\varN$-position (denoted $o(u)=\varN$); otherwise it is a $\varP$-position (denoted $o(u)=\varP$).
We denote by $\varP$ and $\varN$ the sets of all $\varP$-positions and $\varN$-positions, respectively, and by $\varT$ the set of all terminal positions (positions with no legal moves).
All games considered in this paper are \textit{acyclic}, i.e., from any position, no matter how one moves, a terminal position is always reached in finitely many steps.
In this setting, the outcome of a game can be defined recursively as follows:

\begin{definition}\label{def:outcome}

    For any $u\in\varU$:

    \begin{align*}
        o(u)= \varP &\iff \nexists u'\in \nxt(u),\ o(u')=\varN; \\
        o(u)= \varN &\iff \exists  u'\in \nxt(u),\ o(u')=\varP.
    \end{align*}
\end{definition}

We denote by $o^-(u)$ the outcome of $u$ under misère play. Its definition is obtained from the above definition by flipping the outcome of terminal positions:

\begin{definition}\label{def:outcome_mis}

    \begin{align*}
        o^-(u)= \varP &\iff (\nexists u'\in \nxt(u),\ o^-(u')=\varN) \text{~~and~~} u\notin \varT; \\
        o^-(u)= \varN &\iff (\exists  u'\in \nxt(u),\ o^-(u')=\varP) \text{~~or~~}  u\in \varT.
    \end{align*}
\end{definition}

For impartial games, the \textit{SG value} (also called Sprague--Grundy value or Nim-value) provides a more refined characterisation of the eventual outcome of positions; it is defined as follows:
\begin{definition}\label{def:SG}
    For any position $u\in\varU$,
    \[
    \varG(u) \coloneq
    \begin{cases}
        0, & u\in\varT;\\
        \mex(\{ \varG(u') \mid u'\in\nxt(u)\}), & u\notin\varT.
    \end{cases}
    \]
    where the $\mex$ function
    \[
        \mex(S) \coloneq \min (\mathbb{N} \setminus S)
    \]
    denotes the smallest nonnegative integer not appearing in a given set $S$ of nonnegative integers.
\end{definition}

For two impartial games $G_1, G_2$ (with position sets $\varU_1, \varU_2$, respectively),
we define their \textit{disjunctive sum} $G_1 + G_2$, whose positions are combinations of positions from the two games:
\[
    u_1 + u_2,\quad u_1\in\varU_1,\ u_2\in\varU_2,
\]
and whose move rule is: choose one component $u_i$ and make a move in $G_i$ according to its rules, leaving the other component unchanged; the whole position is terminal if and only if both components are terminal.
The following Sprague--Grundy theorem is a cornerstone of combinatorial game theory:

\begin{theorem}\label{thr:sg}

    (Sprague--Grundy theorem) For any impartial game $G$ under normal play:

    (1)~For any position $u\in \varU$,
    \[
        o(u) =
        \begin{cases}
            \varP, & \varG(u)=0; \\
            \varN, & \varG(u)>0.
        \end{cases}
    \]

    (2)~For any positions $u_1, u_2$ of impartial games $G_1, G_2$ under normal play,
    \[
        \varG(u_1+u_2) = \varG(u_1) \oplus \varG(u_2),
    \]
    where $\oplus$ denotes the \textit{bitwise XOR} (nim-sum), i.e., the bitwise exclusive or of the binary expansions of two natural numbers.
\end{theorem}

\subsection{Finite State Automata}

An alphabet $\varSigma$ is a finite set of symbols.
A sequence of symbols of the form
\[
    w=c_1 c_2\ldots c_n, \quad c_n \in\varSigma
\]
is called a \textit{word}, and its length $n$ is denoted by $|w|$.
The word of length $0$ is called the \textit{empty word}, denoted by $\epsilon$.
We denote by
\[
    \operatorname{suffix}_m(w) \coloneq
    \begin{cases}
        \epsilon,           & m=0; \\
        c_1 c_2 \ldots c_m, & 0 < m \leq n; \\
        w,                  & m > n
    \end{cases}
\]
the suffix of $w$ of length $m \in\mathbb{N}$.

We denote by $\varSigma^*$ the set of all words over $\varSigma$ of arbitrary length (including the empty word $\epsilon$), and subsets thereof are called \textit{languages}.
For words $w, v\in\varSigma^*$, we write $wv$ for the word obtained by concatenating them. We write $w^n$ for the word consisting of $n$ repetitions of $w$, and stipulate that $w^0=\epsilon$ for every word $w$.
Clearly $\varSigma^*$ forms a monoid under concatenation (called the \textit{free monoid} generated by $\varSigma$), with identity element the empty word $\epsilon$, which when concatenated with any word yields the original word.

A \textit{nondeterministic finite automaton} (NFA) is a quintuple
\[
    \varA= \langle \varSigma, Q, \delta, Q_0, Q_F \rangle, 
\]
where

(1)~$\varSigma$ is a finite set of \textit{symbols}, called the \textit{alphabet}.

(2)~$Q$ is a finite set of \textit{states}, called the \textit{state set}.

(3)~$\delta$ is a mapping from $Q \times \varSigma$ to the power set of $Q$, called the \textit{transition mapping}.

(4)~$Q_0 \subseteq Q$ is called the \textit{initial state set}.

(5)~$Q_F \subseteq Q$ is called the \textit{final state set} (also referred to as \textit{accepting states} in some literature).

If the transition mapping $\delta$ is deterministic (each state $q$ and symbol $c$ corresponds to a unique next state $\delta(q, c)$),
and $Q_0$ contains exactly one state (in which case this state is usually denoted by $q_0$), then the automaton is called a \textit{deterministic finite automaton} (DFA).
We refer to DFAs and NFAs collectively as FAs.

For convenience, in this paper we relax the definition of a DFA so that the next state corresponding to $(q, c)$ is either unique or empty.
This makes no difference in practice, since we may introduce a \textit{sink state} $q_{\varnothing}$,
treat $\delta(q, c)=\varnothing$ as $\delta(q, c)=q_{\varnothing}$, with all subsequent transitions staying in $q_{\varnothing}$ and never reaching a final state, thereby achieving the same effect.
In actual implementation, we do not need to record the sink state or transitions into it; hence it will be ignored in the remainder of this paper when discussing DFAs.

For any $P, P'\subseteq Q$, $c\in \varSigma$, $w\in\varSigma^*$, we write:
\[
    P\cdot c w = P' w \iff P' = \bigcup_{p\in P} \delta(p, c).
\]

Every FA $\varA$ defines a language over $\varSigma^*$. The languages defined by a DFA and an NFA are respectively
\begin{align*}
    L(\varA) &\coloneq \{w \mid q_0 \cdot w \in Q_F\};\\
    L(\varA) &\coloneq \{w \mid Q_0 \cdot w \cap Q_F \neq \varnothing\}.
\end{align*}

In other words: read the symbols of the word $w= c_1 \ldots c_n$ one by one, applying them on the right to the initial state (or set of states),
and repeatedly perform state transitions according to the mapping $\delta$, until all symbols are consumed (requiring exactly $|w|$ transitions).
For a DFA, the word $w$ is accepted if and only if the final state is a final state;
whereas for an NFA, both the initial state and the states reached at each step are not unique, and $w$ is accepted if and only if the resulting set of states intersects the final state set (i.e., there exists some path reaching a final state).

For example, the following automaton can be used to recognise all words over $\varSigma=\{\x, \e\}$ that contain an odd number of $\e$'s:
\begin{align*}
    \varSigma &= \{\x, \e\}; \\
    Q         &= \{0, 1\}; \\
    \delta    &= \{0\cdot \x = 0,\;\; 1\cdot \x = 1,\;\; 0\cdot \e = 1,\;\; 1\cdot \e = 0\}; \\
    Q_0       &= \{0\}; \\
    Q_F       &= \{1\}.
\end{align*}
The idea is to read the symbols from left to right, using the state to record the parity of the number of $\e$'s, and accept when the final count is odd.

It is easy to see that for a DFA, the effect of $\delta$ is nothing but a right action of the monoid $\varSigma^*$ on the set $Q$.

A language $W\subseteq\varSigma^*$ is called \textit{regular} if it can be recognised by some NFA over the alphabet $\varSigma$.
This paper will make use of many theorems about regular languages and FAs; proofs of these theorems can be found in standard textbooks on formal languages and automata, and we will not elaborate on them here.

\begin{theorem}\label{reg_determ}
    Every regular language (recognisable by an NFA) can also be recognised by some DFA. Given an NFA, there exists an algorithm that constructs an equivalent DFA (recognising the same language).
\end{theorem}

\begin{theorem}\label{reg_finite}
    Every language consisting of only finitely many words is regular.
\end{theorem}

Regular languages are closed under many language operations in a very robust way, meaning that when the inputs are regular languages, the output is again regular, and there exist finite-step algorithms that construct a DFA or NFA recognising the output. For example:

\begin{theorem}\label{reg_closure}
    Regular languages are closed under the following operations, and these closures are effective:

    (1)~Concatenation of regular languages:
    \[
        W_1 W_2 \coloneq \{w_1 w_2 \mid w_1 \in W_1,\ w_2\in W_2\}.
    \]

    (2)~Basic set operations, including intersection $\cap$, union $\cup$, and set difference $-$.
Note that here we do not use $\setminus$ for set difference, in order to distinguish it from left quotient defined below.

    (3)~Kleene plus and star:
    \begin{align*}
        W^+ & \coloneq \{w_1 w_2\ldots w_n \mid w_1, w_2, \ldots, w_n\in W,\ n\in\mathbb{N}^+\}; \\
        W^* & \coloneq \{\epsilon\} \cup W^+.
    \end{align*}

    (4)~Substitution of a factor $v$ by $v'$, where $v,v'\in\varSigma^*$, defined as
    \[
        \mathrm{Sub}_{v\mapsto v'}(W) \coloneq \{x v' y \mid \exists x,y\in\varSigma^*,\ x v y\in W\}.
    \]

    (5)~Left and right quotients of words, defined as
    \begin{align*}
        x \setminus W & \coloneq \{y \mid xy\in W\}; \\
        W ~/~ x     & \coloneq \{y \mid yx\in W\}.
    \end{align*}
\end{theorem}

\section{Partizan Octal Games}
\subsection{Background}

\textit{Octal games} are an important class of ``taking and breaking'' games in combinatorial game theory, discovered and introduced by Richard Guy in 1956~\cite{Octal1956}.
The game is played on several heaps of tokens. Two players alternate moves; on each turn, according to the rules encoded by the game's code, a player removes a certain number of tokens from one heap, and may sometimes split the remaining tokens into two nonempty heaps.
The player who cannot move loses, and the other player wins.

Standard octal games are among the most classical impartial games, i.e., the two players have exactly the same move options;
\textit{partizan octal games}, in contrast, assign different move rules to the two players (usually called Left and Right)~\cite{Part1976}.
Fraenkel and Kotzig first introduced this class of games in 1982~\cite{PartOctal1987}, and proved that their special form, partizan subtraction games, always have periodic outcomes.

In 2004, Plambeck first systematically proposed the misère quotient method for localised analysis of impartial games under misère play~\cite{MisQuot2006}.
Siegel later developed the full theory and suggested that ``a partizan generalization exists''~\cite{PartMisQuot2007}.
Allen attempted to directly extend the notions of indistinguishability and the misère monoid theory from impartial games to the partizan setting~\cite{PartMis2010}.
However, Allen proved that once the game universe contains simple partizan games such as
\[
    1 = \{0 \mid \},
\]
this property no longer holds, which constitutes a fundamental obstacle to generalising the misère quotient method to the partizan case~\cite{PartMisQuot2015}.

Subsequent research on partizan octal games has focused on restricting the game universe to allow for a smooth generalization.
Allen's doctoral thesis and subsequent work found that if the game universe is restricted to ``all-small games'', then the property
\[
    *+* \equiv 0
\]
still holds, providing a viable breakthrough for the generalization~\cite{PartMisQuot2015}.
Milley and Renault studied the universe of ``dead-end games'' and computed its misère monoid~\cite{PartMisQuot2013};
Siegel also published a systematic study of this universe in 2023~\cite{PartMisQuot2023}.
Nowakowski and Ottaway studied a class of vertex-deletion games on graphs and pointed out that these are a special case of partizan octal games~\cite{VertexDel2013}.
That paper explicitly noted that the periodicity study of partizan octal games was still a gap at the time:
``nothing is known about partizan octal games'', and used this as a motivation for studying their graph-theoretic model.

In this chapter, we will define a more general misère quotient by relaxing the algebraic structure, which can overcome the limitations of the game universe and directly perform misère analysis on general partizan octal games.
Using relevant knowledge from automata theory, we design algorithms for searching and verifying misère quotients of games,
which, after validation, can determine the misère quotients of most partizan octal games with relatively simple codes, thereby providing complete solutions to the games.
The remaining unsolved partizan octal games are also expected to be solvable with further increases in computational power.
Based on observations from our results, we propose a ``partizan version of Guy's conjecture''.

\subsection{Definitions and Notation}

We represent a position of an octal game with $n$ heaps by a sequence $u = (a_1, a_2, \ldots, a_n)$ of natural numbers,
where each $a_i \in\mathbb{N}$ denotes the size (number of tokens) of the $i$-th heap.
We denote $\varH^*$ the set of all positions with an arbitrary number of heaps (including the empty position with $0$ heaps).
Octal games are represented by codes of the form
\[
    d_0.d_1 d_2 d_3 \ldots
\]
in octal notation (which gives the games their name), where each $d_s$ indicates the legal moves when removing $s$ tokens:

\begin{center}
    \small
    \begin{tabular}{ll}
        \hline
        $d_s$ & Move \\
        \hline
        0 & no legal move removing $s$ tokens \\
        1 & remove an entire heap of exactly $s$ tokens \\
        2 & remove $s$ tokens, leaving a nonempty remainder \\
        3 & both 1 and 2 are legal \\
        4 & remove $s$ tokens, split the remainder into two nonempty heaps \\
        5 & both 1 and 4 are legal \\
        6 & both 2 and 4 are legal \\
        7 & 1, 2, and 4 are all legal \\
        \hline
    \end{tabular}
    \normalsize
\end{center}

The rules of the game are that the two players alternate taking tokens according to the legal moves encoded; the player who cannot move loses, and the other wins.
We require that the first digit $d_0$ of the code can only be $0$ or $4$ (splitting a heap into two heaps without removing any tokens), since ``removing $0$ tokens'' without splitting is meaningless.

From the rules of octal games, it is easy to see that the order of heaps in a position does not affect the game, i.e., heaps can be regarded as commutative.
Moreover, empty heaps of size $0$ are always unable to move, and adding or removing any number of empty heaps has no effect on the position.

In this paper, we need to introduce a game equivalent to octal games, which we call the \textit{$\e\x$-word representation} of octal games, where positions are represented by strings over the alphabet $\{\e, \x\}$.
Here, $\e$ denotes a token, and $\x$ denotes a separator between heaps; a maximal block of $a$ consecutive $\e$'s surrounded by $\x$'s (of the form $\x\e^a \x$) represents a heap of size $a$.
To avoid ambiguity, we require that the first and last symbols of a position string must be $\x$.
For example:

\begin{center}
    \begin{tabular}{|rcl|rcl|}
        \hline
        $\e\x$-word & & Position & $\e\x$-word & & Position \\
        \hline
        $\x$     & $\leftrightarrow$ & $()$  ~~&~~ $\x\e\e\x$       & $\leftrightarrow$ & $(2)$ \\
        $\x\x$   & $\leftrightarrow$ & $(0)$ ~~&~~ $\x\e\x\e\e\x$   & $\leftrightarrow$ & $(1, 2)$ \\
        $\x\e\x$ & $\leftrightarrow$ & $(1)$ ~~&~~ $\x\e\e\x\x\e\x$ & $\leftrightarrow$ & $(2, 0, 1)$ \\
        \hline
    \end{tabular}
\end{center}

Such a representation uses only 2 symbols to encode all positions of octal games.
At the same time, by expressing token removal and splitting as deleting $\e$'s and inserting $\x$'s, every legal move of an octal game can conveniently be represented as a factor substitution (as defined in~\ref{reg_closure}),
which will be useful for our later proofs:

\begin{center}
    \small
    \begin{tabular}{|rl|rl|}
        \hline
        Move & Factor substitution & Move & Factor substitution \\
        \hline
        $\mathbf{4.}$ & $\e\e\mapsto \e\x\e$       & $\mathbf{0.001}$ & $\x\e\e\e\x\mapsto \x$               \\
        $\mathbf{0.1}$ & $\x\e\x\mapsto \x$        & $\mathbf{0.002}$ & $\x\e\e\e\e\mapsto \x\e$             \\
        $\mathbf{0.2}$ & $\x\e\e\mapsto \x\e$      & $\mathbf{0.004}$ & $\e\e\e\e\e\mapsto \e\x\e$           \\
        $\mathbf{0.4}$ & $\e\e\e\mapsto \e\x\e$    & $\ldots$ & $\ldots$                                     \\
        $\mathbf{0.01}$ & $\x\e\e\x\mapsto \x$     & $\mathbf{0.0}^s\mathbf{1}$ & $\x\e^{s}\x\mapsto \x$     \\
        $\mathbf{0.02}$ & $\x\e\e\e\mapsto \x\e$   & $\mathbf{0.0}^s\mathbf{2}$ & $\x\e^{s}\e\mapsto \x\e$   \\
        $\mathbf{0.04}$ & $\e\e\e\e\mapsto \e\x\e$ & $\mathbf{0.0}^s\mathbf{4}$ & $\e^{s}\e\e\mapsto \e\x\e$ \\
        \hline
    \end{tabular}
    \normalsize
\end{center}

Note that under these rules, the two heaps resulting from a split are always adjacent, which breaks the order-independence of heaps, but does not affect the outcome of the position.

Partizan octal games are obtained from the octal games described above by assigning to the two players different legal moves given by two separate octal codes.
We generally call the player who moves first in a game Left (denoted $\rmL$), and the other player Right (denoted $\rmR$).
In this paper, we will use
\[
    d_{\rmL0}.d_{\rmL1} d_{\rmL2} \ldots : d_{\rmR0}.d_{\rmR1} d_{\rmR2} \ldots
\]
as the code for such a game, with the codes on the left and right of the colon being the move rules for Left and Right, respectively.

Here we give a general method for reducing partizan games to impartial games: define the \textit{label set}
\[
    \varTheta \coloneq \{\rmL, \rmR\},
\]
whose two \textit{labels} denote the move rules of Left and Right, respectively.
For each position
\[
    u = (a_1, a_2, \ldots, a_n)
\]
in the game, we attach a label $\theta \in \varTheta$, obtaining a \textit{labelled position}
\[
    (\theta; u) = (\theta; a_1, a_2, \ldots, a_n),
\]
whose move rule is: move according to the rule corresponding to the label, while simultaneously flipping the label to the other one.
This definition ensures that the resulting game is impartial, and the effect of the partizan game is achieved simply by letting Left always receive positions labelled $\rmL$ and Right always receive positions labelled $\rmR$.
The set of all such labelled positions is denoted by $\varTheta \varH^*$.

Note that the partizan games considered in this paper are still required to be acyclic.
In this setting, every position in an arbitrary partizan game has four possible outcomes. The following table gives the correspondence between these and the outcomes of labelled positions:

\begin{center}
    \begin{tabular}{|l|l|l|}
        \hline
        Unlabelled & Labelled & Meaning \\
        \hline
        $o(u)=\varL$ & $\begin{cases} o(\rmL; u)=\varN; \\ o(\rmR; u)=\varP \end{cases}$ & Left has a winning strategy  \\
        $o(u)=\varR$ & $\begin{cases} o(\rmL; u)=\varP; \\ o(\rmR; u)=\varN \end{cases}$ & Right has a winning strategy \\
        $o(u)=\varN$ & $\begin{cases} o(\rmL; u)=\varN; \\ o(\rmR; u)=\varN \end{cases}$ & the next player has a winning strategy \\
        $o(u)=\varP$ & $\begin{cases} o(\rmL; u)=\varP; \\ o(\rmR; u)=\varP \end{cases}$ & the previous player has a winning strategy \\
        \hline
    \end{tabular}
\end{center}

We also give the $\e\x$-word representation for labelled positions, by placing the label as a symbol at the beginning of the position string:
\begin{center}
    \begin{tabular}{|rcl|rcl|}
        \hline
        $\e\x$-word & & Position & $\e\x$-word & & Position \\
        \hline
        $\rmL\x\e\e\x$ & $\leftrightarrow$ & $(\rmL; 2)$  ~~&~~ $\rmR\x\x\e\x\e\e\x$ & $\leftrightarrow$ & $(\rmR; 1, 2)$ \\
        \hline
    \end{tabular}
\end{center}

\subsection{Misère Quotients and Their Generalization}

The nim-sum decomposition given by the Sprague--Grundy theorem (part (2) of~\ref{thr:sg}) requires games to be played under normal play.
The original motivation for introducing the \textit{misère quotient} was to analyse the outcomes of disjunctive sums of impartial games under misère play,
and it has also been successfully used to solve many (though still only a minority of) misère-play octal games.

We first recall the original definition of the misère quotient (from~\cite{MisQuot2006}):

\begin{definition}
    Let a game $G$ have positions closed under disjunctive sum (i.e., $G+G=G$),
    and let $u_1, u_2 \in \varU$ be two positions of $G$, and $\varV\subseteq \varU$ be a set of positions also closed under disjunctive sum $+$ (sometimes also required to be closed under $\nxt(\cdot)$).
    If for every $v\in \varV$ we have
    \[
        o^-(u_1 + v) = o^-(u_2 + v),
    \]
    then $u_1, u_2$ are said to be equivalent modulo $\varV$, denoted
    \[
        u_1 \equiv_{\varV} u_2.
    \]
    The relation $\equiv_{\varV}$ thus defined is always an equivalence relation,
    and the set of equivalence classes $\varU / \cong_{\varV}$ of all positions under this relation always forms a commutative monoid, called the \textit{misère quotient} of $\varU$ modulo $\varV$.
\end{definition}

Clearly, positions in the same equivalence class must have the same outcome under misère play (taking $v$ to be the empty heap suffices to prove this), so the equivalence classes can also be divided into $\varN$-classes and $\varP$-classes according to their outcomes.
In general, the misère quotient we most wish to obtain is the one with $\varV$ taken to be the whole $\varU$ (in which case $\equiv_{\varV}$ is usually denoted simply by $\equiv$),
because such a misère quotient often provides a comprehensive characterisation of the outcomes of all positions of the game.
Taking misère-play octal games as an example, we can repeatedly decompose a position into the form $u_1 + v$, then find a simpler $u_2$ equivalent to $u_1$ and reduce the position to $u_2 + v$;
by continuing this reduction, the position can eventually be reduced to the simplest representative of its equivalence class, thereby revealing its equivalence class.
If we have characterised the structure of $\varU / \equiv_{\varV}$ and the outcome associated with each equivalence class in sufficient detail, then the outcome of any position can be determined, and the game can be regarded as completely solved.

For unlabelled positions of partizan octal games under normal play, the above definition can also directly yield a misère quotient (simply by replacing $o^-(\cdot)$ with $o(\cdot)$).
However, as Allen pointed out~\cite{PartMis2010}, for most partizan octal games, the structure of such a misère quotient is too unwieldy to yield a solution of the game.
For labelled positions, the addition above is not closed: the rule label of a game must be \textit{global}, but a position obtained by adding two positions has the form
\[
    (\theta_a; a_1, \ldots, a_m) + (\theta_b; b_1, \ldots, b_n),
\]
whose two components each belong to a label; flipping the label only changes that component rather than the global rule, so this is not a legal position.

Therefore, in this paper we define a more general misère quotient: we relax the structure directly from a commutative monoid to a noncommutative monoid acting on a set:

\begin{definition}\label{def:s_serie}
    Let $U$ be a set and $V$ a monoid (with multiplication as the operation). We say that $V$ acts (on the right)
    \footnote{In other literature, monoid actions on sets are usually defined as left actions; in this paper we define them as right actions for convenience. The two differ only in the order of writing.}
    on $U$ if each $v\in V$ corresponds to a self-map $\delta_v: U\rightarrow U$ on $V$, satisfying

    (1)~$\forall v_1, v_2 \in V,\ \forall u\in U,\ \delta_{v_2}(\delta_{v_1}(u))= \delta_{v_1 v_2}(u)$;

    (2)~the identity of $V$ corresponds to the identity map.

    We write $\delta_v (u)$ as $u \cdot v$, and $\delta_{v_2}(\delta_{v_1}(u))$ as $u \cdot v_1 v_2$.
\end{definition}

\begin{definition}\label{def:expanded_misquot}
    Let $u_1, u_2 \in \varU$ be two positions of a game, and let a monoid $V$ act (on the right) on $\varU$.
    If for every $v\in V$ we have
    \[
        o(u_1 \cdot v) = o(u_2 \cdot v),
    \]
    then $u_1, u_2$ are said to be equivalent modulo $V$, denoted
    \[
        u_1 \cong_{V} u_2.
    \]
    It is easy to see that $\cong_{V}$ thus defined is always an equivalence relation,
    and the set of equivalence classes $\varU / \cong_{V}$ of all positions under this relation is the \textit{generalized misère quotient} of $\varU$ modulo $V$.
\end{definition}

Note that at this point $\varU$ has no semigroup operation defined internally, so such a generalized misère quotient is no longer a semigroup.
However, by definition, its elements are closed under the right action of elements of $V$ (the $\cdot$ operation), which can be seen as a weakening of the closure under semigroup multiplication.

For the labelled positions of partizan octal games as defined above, we can naturally define a right action of unlabelled positions on them.
Take $\varU=\varTheta\varH^*$ and $V=\varH^*$ above; for any $(\theta; u_1, \ldots, u_m)\in \varU$ and $(v_1, \ldots, v_n) \in V$,
the right action is naturally defined as concatenation of heaps:
\[
    (\theta; u_1, \ldots, u_m) \cdot (v_1, \ldots, v_n) \coloneq (\theta; u_1, \ldots, u_m, v_1, \ldots, v_n),
\]
which yields a generalized misère quotient for labelled positions of partizan octal games.

\subsection{Algorithmic Principles}

The key to our approach for solving partizan octal games lies in the following assumption:
we assume that the $\e\x$-words of all $\varN$-positions (or all $\varP$-positions) among the labelled positions always form a regular language.
On this basis, we design two algorithms using relevant principles from automata theory,
one for searching for a possible DFA recognising this regular language, and the other for proving that the DFA we find is indeed correct.
We call the DFA recognising all $\varN$-positions the $\varN$-DFA, and the one recognising all $\varP$-positions the $\varP$-DFA; together they are referred to as the \textit{outcome DFA} of the game.

\subsubsection{Searching for the Outcome DFA}

In this section we discuss how to search for such a DFA.

We set $\varSigma =\{\x, \e\}$.
Since we use a labelled character representation, the first symbol of a position is always a label; hence the label set $\varTheta=\{\rmL, \rmR\}$ must also be included in the alphabet of the DFA. Thus the alphabet of the DFA is $\varSigma \cup \varTheta$.

The difficulty, of course, lies in setting up the states and transitions. If we already knew the full structure of some regular language $W$, we could use the following theorem to construct the corresponding DFA:

\begin{lemma}\label{build_fa}
    For an alphabet $\varSigma$ and a regular language $W\subseteq \varSigma^*$, define an equivalence relation for $w_1, w_2\in \varSigma^*$ by
    \[
        w_1 \cong w_2 \iff \forall v\in \varSigma^*,\ (w_1 v\in W \leftrightarrow w_2 v\in W),
    \]
    then a DFA recognising $W$ can be given as follows (where $\overline{w}$ denotes the equivalence class of $w$):
    \begin{align*}
        Q &= W / \cong; \\
        \delta(q, c) &= \{\overline{w_q c} \mid w_q\in q,\ c\in\varSigma\}; \\
        Q_0 &= \overline{\epsilon} \\
        Q_F &= \{ \overline{w} \mid w\in W\}.
    \end{align*}

    Moreover, for any language $W\subseteq \varSigma^*$, the DFA thus constructed is always minimal (has the fewest states).
\end{lemma}

However, we do not have direct access to the outcomes of all positions; we can only compute the outcomes of small positions.
Therefore, we need to take $V\subseteq\varSigma^*$ and weaken the above $\cong$ to $\cong_{V}$ defined as follows:
\[
    u_1 \cong_V u_2 \iff \forall v\in V,\ (u_1 v\in W \leftrightarrow u_2 v\in W),
\]
and still construct the corresponding DFA in the same way. Clearly, the larger $V$ is, the more likely the resulting DFA is the correct one.

In the concrete algorithm, we take one representative word from each equivalence class of $Q$ (so that such words are in bijection with states), and determine whether another word belongs to that equivalence class by testing its equivalence with the representative.
Since we have stipulated that only positions of the form $\varTheta\x\varSigma^*\x$ are legal, we can directly discard prefixes that cannot form legal positions.
The three prefixes $\epsilon, \rmL, \rmR$ are special (since they do not themselves form positions), and we treat them as three separate equivalence classes;
we only need to partition the remaining words beginning with $\rmL\x$ or $\rmR\x$ into equivalence classes.

As for the choice of $V$, first we may restrict it to words ending in $\x$ to ensure legality of positions. Although this means $V$ no longer covers the whole $\varSigma^*$, the practical effect is simply to discard certain illegal positions whose outcomes are meaningless anyway.
On this basis, since positions differing by a permutation or by any number of empty heaps are naturally equivalent, we can omit such positions to save computation.
Specifically, we choose words of the following form from $\varSigma^*$ as elements of $V$:
\begin{align*}
    v &= \e^a \x \e^{b_1}\x \e^{b_2}\x \ldots \x \e^{b_n}\x, \\
      &~~~~ b_1\geq b_2\geq\ldots\geq b_n >0, |v|\leq n_v,
\end{align*}
where $n_v$ is the maximum length we select for $V$. A larger $n_v$ makes it more likely to obtain the correct DFA, but the computational cost grows exponentially with it.

At the start of the procedure, we initialise the state table as
\[
    \{\epsilon, \rmL, \rmR, \rmL\x, \rmR\x\},
\]
where the first three are special states. Moreover, if the positions corresponding to the last two states (both of which are clearly terminal) satisfy the prescribed outcome, they are marked as final states.
Then we set up a queue of words to be reduced:
\[
    \{\rmL\x\x,\ \rmL\x\e,\ \rmR\x\x,\ \rmR\x\e\}.
\]
As the procedure runs, it repeatedly pops the first word $q=\theta\x wc$ (where $c$ is the last character of the word) from the queue and tests whether it is equivalent to some existing non-special representative $q'$:

(1)~If such an existing state exists, the reduction succeeds; the word does not correspond to a new state, and we merely add a transition given by this equivalence:
\[
    \theta\x w \cdot c = q'.
\]
No further reduction is needed for words having this state as a prefix.

(2)~If the reduction fails, the word corresponds to a new state and is added to the state table. If this state also ends with $\x$ and the corresponding position satisfies the prescribed outcome, it is marked as a final state.
In this case, words having $q$ as a prefix still need to be reduced, so we append $q\x$ and $q\e$ to the end of the queue, and add the \textit{trivial transitions}
\[
    \begin{cases}
        q\cdot \x = q\x; \\
        q\cdot \e = q\e.
    \end{cases}
\]

When the queue of words to be reduced becomes empty, all words have been assigned to equivalence classes represented by the words currently in the state table, and the procedure terminates.
If the number of states becomes too large and the procedure has still not terminated, it is treated as an overflow.

From the above process, it is not hard to see that the state tables and transition mappings of the $\varN$-DFA and the $\varP$-DFA are identical, differing only in their final states: the two sets of final states are disjoint and together contain all states ending with $\x$.
Accordingly, once we have the $\varN$-DFA, we can immediately obtain the corresponding $\varP$-DFA.

\subsubsection{Verification of the Outcome DFA}

For both normal and misère play, we only need to check whether the candidate outcome DFA is correct according to Definition~\ref{def:outcome} and Definition~\ref{def:outcome_mis}, respectively.
For a set of positions $\varW \subseteq \varU$, we define
\begin{align*}
    \nxt(\varW) &\coloneq \{w' \mid w\in\varW,\ w' \in \nxt(w)\}; \\
    \prv(\varW) &\coloneq \{w \mid \exists w'\in\nxt(w),\ w' \in \varW \},
\end{align*}

Then we can rewrite the definitions of $\varP$- and $\varN$-positions in set-theoretic form, yielding:

\begin{corollary}\label{outcome_cond}
    For any sets of positions $\varW_{\varP}, \varW_{\varN} \subseteq \varU$ of an impartial game under normal play,
    they are exactly the sets of all $\varP$- and $\varN$-positions of the game if and only if all of the following conditions hold:

    (1)~$\varT \subseteq \varW_{\varP}$.

    (2)~$\nxt(\varW_{\varP}) \subseteq \varW_{\varN}$.

    (3)~$\varW_{\varN} \subseteq \prv(\varW_{\varP})$.
\end{corollary}

To verify whether the found DFA is correct, it suffices to use automaton operations to check whether the above conditions hold.

Under our $\e\x$-word representation, both $\nxt(\cdot)$ and $\prv(\cdot)$ can be expressed as factor substitutions on languages,
from which one can construct FAs recognising $\nxt(\varW_{\varP})$ and $\prv(\varW_{\varP})$.
Furthermore, the set of terminal positions $\varT$ is given by all positions that contain no movable factor (i.e., no factor that can be substituted). Denoting the set of all such factors by $\varV_{\mathrm{m}}$, we have
\[
    \varT = \varU - (\varSigma^* \varV_{\mathrm{m}} \varSigma^*),
\]
where $\varV_{\mathrm{m}}$ contains only finitely many words; hence by Theorems~\ref{reg_finite} and~\ref{reg_closure}, this is clearly a regular language.
In summary, it is feasible to use automaton operations to check whether the conditions in Corollary~\ref{outcome_cond} hold.

For misère play, we simply replace (1) and (3) in Corollary~\ref{outcome_cond} with, respectively:

($1^-$)~$\varT \subseteq \varW_{\varN}$.

($3^-$)~$(\varW_{\varN} - \varT) \subseteq \prv(\varW_{\varP})$.

That is, terminal positions are $\varN$-positions, and they need not have a $\varP$-position as a successor.

\subsection{Computational Results}

\subsubsection{$\mathbf{0.04\col0.03}$ as an Example}

We first take $\mathbf{0.04\col0.03}$ as an example to illustrate what information can be obtained from the outcome DFA produced by our algorithm.
In this game, Left's legal move is to remove two tokens from a heap and split the remainder into two nonempty heaps; Right's legal move is to remove two tokens (possibly leaving nothing) without splitting.
Traditional methods are powerless against this game, but our Python implementation of the algorithm successfully searches for and verifies the $\varN$-DFA of this game within 2.4 seconds, yielding
\[
    A_{\varN}=\langle Q, \varSigma \cup \varTheta, \delta, \epsilon, Q_{\varN} \rangle,
\]
where:

(1)~The state set $Q$ contains 25 states (excluding the sink state). Grouped by their first and last symbols, they are:
\small
\begin{align*}
    \text{Special:~~}  & \epsilon, \rmL, \rmR; \\
    \text{$\rmL\x$-class:~~} & \rmL\x, \rmL\x\e^2\x, \rmL\x\e^4\x; \\
    \text{$\rmL\e$-class:~~} & \rmL\x\e, \rmL\x\e^2, \rmL\x\e^3, \rmL\x\e^2\x\e, \rmL\x\e^4, \rmL\x\e^5, \\
                   & \rmL\x\e^4\x\e, \rmL\x\e^4\x\e^2, \rmL\x\e^4\x\e^3, \rmL\x\e^4\x\e^4, \rmL\x\e^4\x\e^5; \\
    \text{$\rmR\x$-class:~~} & \rmR\x, \rmR\x\e^2\x; \\
    \text{$\rmR\e$-class:~~} & \rmR\x\e, \rmR\x\e^2, \rmR\x\e^3, \rmR\x\e^2\x\e, \rmR\x\e^4, \rmR\x\e^5.
\end{align*}
\normalsize
Each state in the Lx-class and Rx-class corresponds to an irreducible labelled position; all other positions are eventually reduced via transitions to one of these irreducible positions.
Therefore, each state in the Lx-class and Rx-class corresponds to an equivalence class in the generalized misère quotient, and the word of the state itself is a shortest representative of that class.

Thus, the generalized misère quotient of $\mathbf{0.04\col0.03}$ contains 5 classes. Rewriting them back to the position format gives:
\[
    (\varU/\cong) = \{ \overline{(\rmL;)}, \overline{(\rmL; 2)} , \overline{(\rmL; 4)}, \overline{(\rmR;)}, \overline{(\rmR; 2)} \}.
\]
The DFA $A_{\varN}$ has 2 final states:
\[
    Q_{\varN} = \{\rmL\x\e^4\x,\ \rmR\x\e^2\x\}.
\]
So among the above 5 classes, 2 are $\varN$-classes:
\[
    (\varN/\cong) = \{ \overline{(\rmL; 4)}, \overline{(\rmR; 2)} \}.
\]
The remaining 3 are $\varP$-classes.

(2)~The transition mapping $\delta$ contains 48 transitions in total, of which 30 are trivial transitions of the form
\[
    p\cdot \x = p \text{~~or~~} p\cdot \e = p\e.
\]
The remaining 18 non-trivial transitions (i.e., reduction transitions) are also grouped by class as follows:
\small
\begin{align*}
    \text{$\rmL\x$-class:~~} & \rmL\x\e\cdot\x=\rmL\x,~~ \rmL\x\e^3\cdot\x=\rmL\x\e^2\x,~~ \rmL\x\e^5\cdot\x=\rmL\x\e^2\x,~~ \\
         & \rmL\x\e^4\x\e\cdot\x=\rmL\x\e^4\x,~~ \rmL\x\e^4\x\e^2\cdot\x=\rmL\x\e^2\x,~~ \rmL\x\e^4\x\e^3\cdot\x=\rmL\x\e^2\x, \\
         & \rmL\x\e^4\x\e^4\cdot\x=\rmL\x\e^2\x,~~ \rmL\x\e^4\x\e^5\cdot\x=\rmL\x\e^2\x,~~ \rmL\x\e^2\x\e\cdot\x=\rmL\x\e^2\x; \\
    \text{$\rmL\e$-class:~~} & \rmL\x\e^5\cdot\e=\rmL\x\e,~~ \rmL\x\e^4\x\e^5\cdot\e=\rmL\x\e^4\x\e,~~ \rmL\x\e^2\x\e\cdot\e=\rmL\x\e^2\x\e; \\
    \text{$\rmR\x$-class:~~} & \rmR\x\e\cdot\x=\rmR\x,~~ \rmR\x\e^4\cdot\x=\rmR\x\e^2\x,~~ \rmR\x\e^5\cdot\x=\rmR\x\e^2\x, \\
         & \rmR\x\e^2\x\e\cdot\x=\rmR\x\e^2\x; \\
    \text{$\rmR\e$-class:~~} & \rmR\x\e^5\cdot\e=\rmR\x\e,~~ \rmR\x\e^2\x\e\cdot\e=\rmR\x\e^2\x\e.
\end{align*}\normalsize
These reduction transitions completely characterise the equivalence relation of the generalized misère quotient.
It is worth noting that our program does not require each word to reduce to a word with the same label; however, from the actual results, only a very few partizan octal games (e.g., $\mathbf{4.7 \mid 0.75}$) have positions with different labels that are equivalent to each other.

We can also rewrite the $\e\x$-words in the above transitions back to the position format.
In doing so, each $\theta\x$-class gives a pair of equivalent positions, while each $\theta\e$-class gives a family of one-to-one equivalences:
\small
\begin{align*}
    \text{$\rmL\x$-class:~~} & (\rmL; 1)\cong(\rmL;),~~  (\rmL; 3)\cong (\rmL; 2),~~ (\rmL; 5)\cong (\rmL; 2),~~ \\
         & (\rmL; 4, 1)\cong (\rmL; 4),~~ (\rmL; 4, 2)\cong (\rmL; 2),~~ (\rmL; 4, 3)\cong (\rmL; 2), \\
         & (\rmL; 4, 4)\cong (\rmL; 2),~~ (\rmL; 4, 5)\cong (\rmL; 2),~~ (\rmL; 2, 1)\cong (\rmL; 2); \\
    \text{$\rmL\e$-class:~~} & (\rmL; 5, 1\!+\!i)\cong (\rmL; 1\!+\!i),~~ (\rmL; 4, 6\!+\!i)\cong(\rmL; 5\!+\!i),~~ (\rmL; 2, 2\!+\!i)\cong (\rmL; 2,1\!+\!i); \\
    \text{$\rmR\x$-class:~~} & (\rmR; 1)\cong(\rmR),~~ (\rmR; 4)\cong (\rmR; 2),~~ (\rmR; 5)\cong (\rmR; 2), \\
         & (\rmR;2,1,v) \cong (\rmR; 2,v); \\
    \text{$\rmR\e$-class:~~} & (\rmR;6\!+\!i)\cong(\rmR;1\!+\!i),~~ (\rmR;2,2\!+\!i)\cong(\rmR;2,1\!+\!i).
\end{align*}\normalsize
where $i\in\mathbb{N}$ is an arbitrary natural number.

\begin{proposition}\label{finite_misquot}
    If a partizan octal game (or any similar partizan game with multi-heap positions) admits an outcome DFA, then its generalized misère quotient is finite.
\end{proposition}

\begin{proof}
    From the correspondence shown above, the final states of the $\varP$- and $\varN$-DFAs are in bijection with the $\varP$- and $\varN$-classes in the generalized misère quotient; hence finiteness of the former implies finiteness of the latter.
\end{proof}

The converse of this proposition need not hold, since one can imagine a game whose outcome cannot be determined by a DFA, but can be determined by an automaton resembling a DFA with infinitely many states yet only finitely many final states.

Moreover:

\begin{proposition}\label{normal_per}
    If a partizan octal game (or any similar partizan game with multi-heap positions) admits an outcome DFA, then its outcomes are \textit{ultimately periodic}, i.e.:
    there exist $s, t\in\mathbb{N}^+$, $s\geq t$,
    such that for every $(\theta; a_1, a_2,\ldots, a_n)\in\varTheta\varH^*$ with $a_1 > s$, we have
    \[
        o(\theta; a_1, a_2, \ldots, a_n)=o(\theta; a_1-t, a_2, \ldots, a_n).
    \]
    The same holds under misère play.
\end{proposition}

\begin{proof}
    Since the outcome DFA has finitely many states, for any state $q\in Q$, the sequence of states
    \[
        q\cdot \epsilon,\ q\cdot \e,\ q\cdot \e\e,\ q\cdot\e\e\e,\ \ldots,\ q\cdot\e^n ,\ldots
    \]
    must be ultimately periodic, i.e., there exist $s_q, t_q\in\mathbb{N}^+$, $s_q\geq t_q$, such that
    \[
        q\cdot \e^{s_q+1} = q\cdot \e^{s_q+1-t_q},
    \]
    and hence for every positive integer $a>s_q$ we have
    \[
        q\cdot \e^a = q\cdot \e^{a-t_q}.
    \]

    It suffices to take
    \begin{align*}
        t &= \operatorname{lcm} \{t_q \mid q\in Q\};\\
        s &= \max(t_{\theta},\ \max \{s_q \mid q\in Q\})
    \end{align*}
    to ensure that for every $a>s$,
    \[
        q\cdot \e^a = q\cdot \e^{a-t}
    \]
    holds for every state $q\in Q$, so that $\e^a$ and $\e^{a-t}$ are completely equivalent in the outcome DFA, yielding periodic outcomes.
\end{proof}

\subsubsection{Other Results}

Research on impartial octal games generally focuses on \textit{three-digit octal games} of the form $\mathbf{0.xxx}$ or $\mathbf{4.xx}$.
However, when partizan games are included, the number of such games (excluding those where one player has no legal moves) amounts to 165,600 up to symmetry, making the computational cost prohibitive.
Therefore, we attempt only to solve all \textit{two-digit} games, i.e., those whose Left and Right codes are both of the form $\mathbf{0.xx}$ or $\mathbf{4.x}$.
There are 2,556 such impartial and partizan octal games up to symmetry.
To eliminate symmetry, we consider only those where Left's octal code is lexicographically no smaller than Right's.

We set the maximum number of states to 500 and the maximum length of test suffixes to 30, and applied our algorithm to each of these 2,556 games one by one.
The results are shown in Figure~\ref{fig:PartOctal}.
Among them, 258 games failed to yield an outcome DFA due to state count or memory overflow, 78 games produced DFAs that failed verification, and the remaining 2,220 games were successfully solved (86.85\%).

\begin{figure}
    \centering
    \includegraphics[width=0.8\linewidth]{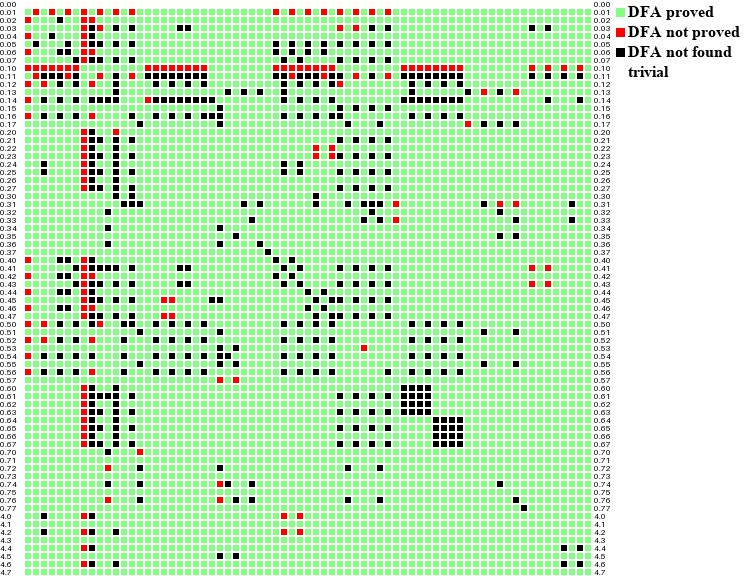}
    \caption{Results for two-digit partizan octal games}
    \label{fig:PartOctal}
\end{figure}

It is worth noting that some impartial games among them (lying on the diagonal of the figure) actually have complete solutions via the Sprague--Grundy theorem, yet our program failed to solve them.
This suggests that some of the unsolved partizan octal games may also be solvable with sufficient computational power.

We group games with equivalent outcome DFAs into \textit{outcome equivalence classes}, obtaining 456 classes in total.

We adopt the shorthand notation for octal game codes from~\cite{MisQuot2006}:
\begin{align*}
    \dot{d}  &\coloneq d, d+1; \\
    \ddot{d} &\coloneq d, d+2; \\
    \bar{d}  &\coloneq d, d+1, d+2, d+3.
\end{align*}
The following table lists the misère quotient statistics for the top 10 outcome equivalence classes with the largest state count $|Q|$:

\begin{center}
    \begin{tabular}{|c|cccc|c|c|}
        \hline
        Code & $|\overline{\varP_\rmL}|$ & $|\overline{\varN_\rmL}|$ & $|\overline{\varP_\rmR}|$ & $|\overline{\varN_\rmR}|$ & $|\overline{\varU}|$ & $|Q|$ \\
        \hline
        $\mathbf{0.74\col0.31}$                             & 52 & 21 & 21 & 52 & 146 & 443 \\
        $\mathbf{0.76\col0.31}$                             & 52 & 22 & 21 & 51 & 146 & 439 \\
        $\mathbf{0.7\col0.17}$                              & 17 & 21 & 21 & 17 & 76  & 375 \\
        $\mathbf{4.7\col0.73}$                              & 4  & 15 & 12 & 12 & 43  & 316 \\
        $\mathbf{0.44\col0.12}$                             & 4  & 7  & 3  & 3  & 17  & 310 \\
        $\mathbf{0.72\col0.13},\ \mathbf{0.76\col0.13}$      & 29 & 22 & 22 & 29 & 102 & 309 \\
        $\mathbf{4.\ddot{0}\col0.4\ddot{1}}$                & 14 & 15 & 15 & 14 & 58  & 285 \\
        $\mathbf{0.6\ddot{3}\col0.1},\ \mathbf{0.67\col0.1}$ & 29 & 15 & 15 & 29 & 88  & 261 \\
        $\mathbf{0.4\ddot{4}\col0.4\ddot{0}}$               & 1  & 4  & 9  & 10 & 24  & 251 \\
        $\mathbf{0.74\col0.54}$                             & 1  & 4  & 8  & 10 & 23  & 232 \\
        \hline
    \end{tabular}
\end{center}

The following table lists the 10 outcome equivalence classes with the smallest state count
\footnote{The ellipses $\ldots$ indicate partizan games included in the class; the number following is the total count of games in that outcome equivalence class.}:
\begin{center}
    \begin{tabular}{|c|cccc|c|c|}
        \hline
        Code & $|\overline{\varP_\rmL}|$ & $|\overline{\varN_\rmL}|$ & $|\overline{\varP_\rmR}|$ & $|\overline{\varN_\rmR}|$ & $|\overline{\varU}|$ & $|Q|$ \\
        \hline
        $\mathbf{0.13\col0.02},\ \ldots$ (296 games)      & 1  & 1  & 1  & 0  & 3   & 8 \\
        $\mathbf{0.2\col0.13},\ \ldots$ (198 games)       & 1  & 0  & 1  & 1  & 3   & 8 \\
        $\mathbf{0.02\col0.02}$                     & 2  & 1  & 2  & 1  & 3   & 10 \\
        $\mathbf{0.03\col0.03}$                     & 2  & 1  & 2  & 1  & 3   & 10 \\
        $\mathbf{0.12\col0.12}$                     & 2  & 1  & 2  & 1  & 3   & 10 \\
        $\mathbf{0.13\col0.13}$                     & 2  & 1  & 2  & 1  & 3   & 10 \\
        $\mathbf{0.1\col0.1}$                       & 2  & 1  & 2  & 1  & 3   & 10 \\
        $\mathbf{0.01\col0.01}$                     & 2  & 1  & 2  & 1  & 3   & 12 \\
        $\mathbf{0.11\col0.11}$                     & 2  & 1  & 2  & 1  & 3   & 12 \\
        $\mathbf{0.32\col0.01},\ \ldots$ (224 games)      & 1  & 1  & 2  & 1  & 5   & 12 \\
        $\mathbf{0.41\col0.32},\ \ldots$ (32 games)       & 2  & 1  & 1  & 1  & 5   & 12 \\
        $\mathbf{0.75\col0.75},\ \mathbf{4.7\col0.75},\ \mathbf{4.7\col4.7}$ & 2  & 3  & 2  & 3  & 5   & 12 \\
        \hline
    \end{tabular}
\end{center}

From these results, we observe that most partizan octal games have finite generalized misère quotients,
and in comparison, they seem to tend to have smaller (generalized) misère quotients than their impartial counterparts.
One possible explanation is that asymmetric rules are likely to give one player an advantage, making positions easier to reduce.
Based on this, we propose the following ``partizan version of Guy's conjecture'':

\begin{problem}
    Do all impartial or partizan octal games admit an outcome DFA?
\end{problem}

It is worth noting that the impartial case of this problem does not imply the original Guy's conjecture, since it only yields periodicity of outcomes rather than periodicity of SG values.

We also attempted to solve the corresponding games under misère play. As in the impartial case, the behaviour of partizan octal games under misère play is significantly worse than under normal play.
The results are shown in Figure~\ref{fig:MisPartOctal}.
Among 2,485 games, 607 failed to yield an outcome DFA, 182 produced DFAs that failed automatic verification, and only 1,696 games were successfully solved (68.25\%).

\begin{figure}
    \centering
    \includegraphics[width=0.8\linewidth]{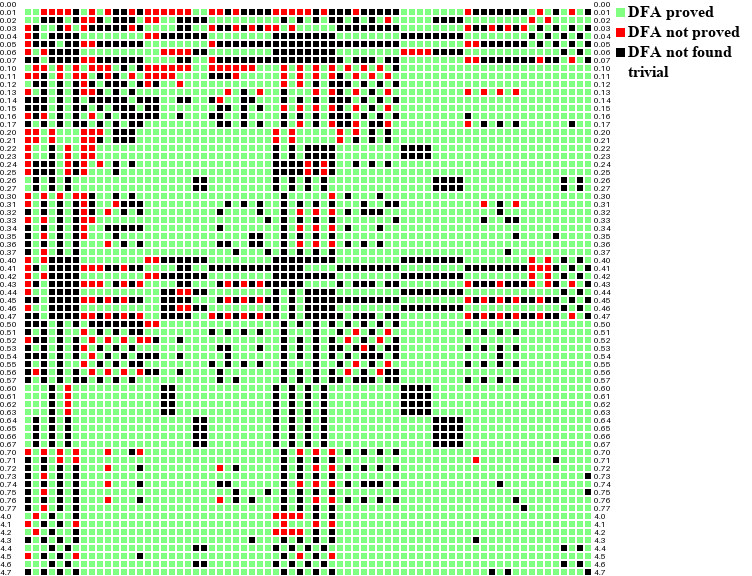}
    \caption{Results for two-digit partizan octal games under misère play}
    \label{fig:MisPartOctal}
\end{figure}

\subsubsection{Algebraic Periodicity}

Concerning misère octal games, it is worth discussing the phenomenon of ``algebraic periodicity'' mentioned in~\cite{MisQuot2006}:

``
When this regularity continues indefinitely, the full quotient can be deduced from a finite number of partial quotients;
but unlike in the ``stable'' case, the full quotient is infinite, whereas every partial quotient is finite.
We've christened this behavior ``algebraic periodicity,'' but we can't give
a precise definition because we don't fully understand how to describe it in general.
''

That reference gives the following examples of impartial games exhibiting ``algebraic periodicity'':
\[
    \mathbf{0.26},\ \mathbf{4.7},\ \mathbf{0.54},\ \mathbf{0.145},
    \mathbf{0.157},\ \mathbf{0.175},\ \mathbf{0.355},\ \mathbf{0.357},\ \mathbf{0.516},\ \mathbf{0.724},\ \mathbf{0.734}.
\]
Among these, $\mathbf{0.26}$ and $\mathbf{4.7}$ have been successfully solved by Allemang~\cite{1984Algper} (with some errors corrected in~\cite{MisQuot2006}),
while $\mathbf{0.54}$ and $\mathbf{0.145}$ are observed to have more complex behaviour than the other examples.

Here we attempt to give a rigorous definition of ``algebraic periodicity''; based on results from smaller-scale computations, it can encompass all of the above examples except the two more complex cases $\mathbf{0.54}$ and $\mathbf{0.145}$:

\begin{definition}

    We say that a game is \textit{algebraically periodic} if there exist natural numbers $p_1, p_2\in\mathbb{N}$ and positive integers $t_1, t_2\in\mathbb{N}^+$,
    such that for every position (with heaps sorted in non-increasing order) with largest heap $a_1$ sufficiently large, i.e., $a_1\geq p_1+t_1$,
    \[
         u=(a_1, a_2, \ldots, a_n)\in\varU,~~ a_1\geq a_2\geq \ldots \geq a_n,
    \]
    it can always be reduced to a smaller position in one of the following two ways:

    (1)~\textit{Single-heap reduction}: if the second-largest heap $a_2 < p_2+t_2$ (in particular, when $n=1$ we regard $a_2=0 < p_2+t_2$), then
    \[
        o^-(u)=o^-(a_1-t_1, a_2, \ldots, a_n).
    \]
    (2)~\textit{Two-heap reduction}: if $a_2 \geq p_2+t_2$, then
    \[
        o^-(u)=o^-(a_1-t_2, a_2-t_2, a_3 \ldots, a_n).
    \]

    We call the lexicographically smallest tuple $(p_1, t_1, p_2, t_2)$ the \textit{parameters} of algebraic periodicity,
    where $(p_1, t_1)$ and $(p_2, t_2)$ can be viewed as the $(\text{preperiod}, \text{period})$ under the two reduction schemes, respectively.
\end{definition}

It is easy to see that any position can be reduced by repeatedly applying the two reductions until every heap is smaller than $p_1+t_1$.
According to~\cite{MisQuot2006}, such positions always have a finite misère quotient, and hence are solvable.

Figure~\ref{fig:algper_oc_images} shows outcome images for some algebraically periodic octal games, where the red and cyan shaded regions indicate positions to which the single-heap reduction and the two-heap reduction apply, respectively.
For a position $u=(a_1, a_2, \ldots, a_n)$, we define the difference between its largest heap and the sum of the remaining heaps as
\[
    \Delta(u) \coloneq a_1-a_2-\ldots-a_n,\ \text{where~~}a_1=\max(u).
\]
If this difference is sufficiently large, then no matter how one applies the two-heap reduction, the largest heap will always retain too many tokens, forcing a subsequent single-heap reduction;
conversely, only two-heap reductions are needed.
Thus, positions in algebraically periodic games can often be divided into the following three categories, each with different patterns:

(1)~``Corner'' positions: $\max(u)$ is sufficiently small.

(2)~``Inner'' positions: $\max(u)$ is sufficiently large and $\Delta(u)$ is sufficiently small.

(3)~``Outer'' positions: $\max(u)$ is sufficiently large and $\Delta(u)$ is sufficiently large.

\begin{figure}
    \centering
    \includegraphics[width=0.9\linewidth]{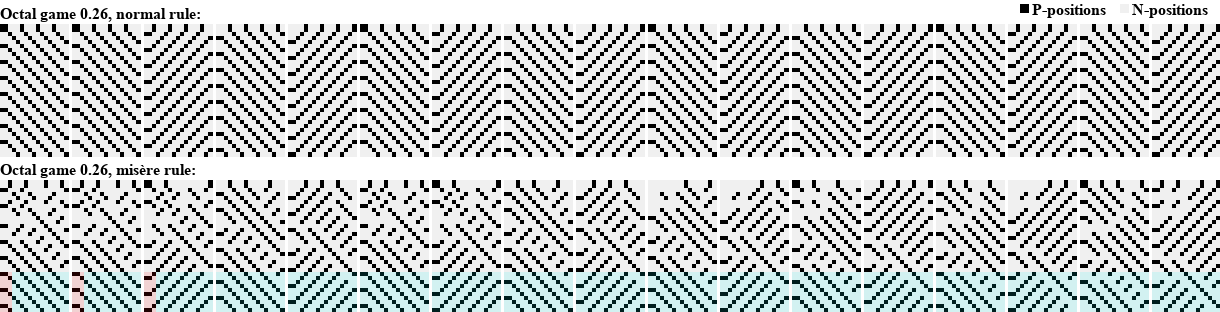}
    \includegraphics[width=0.9\linewidth]{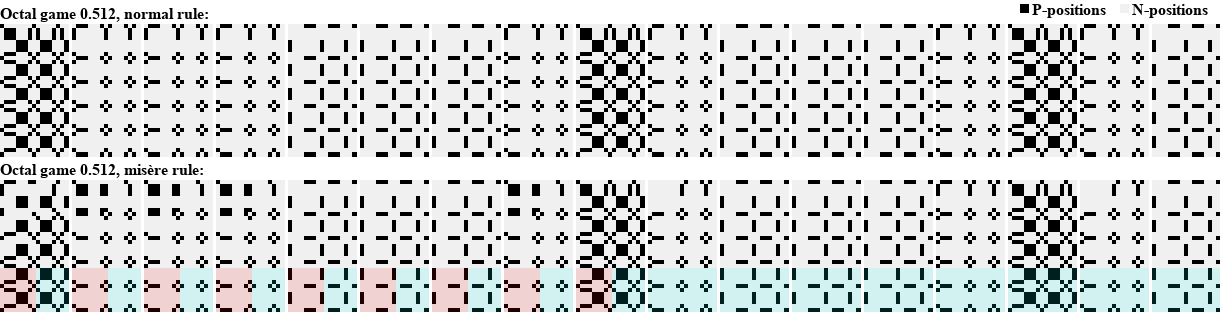}
    \includegraphics[width=0.9\linewidth]{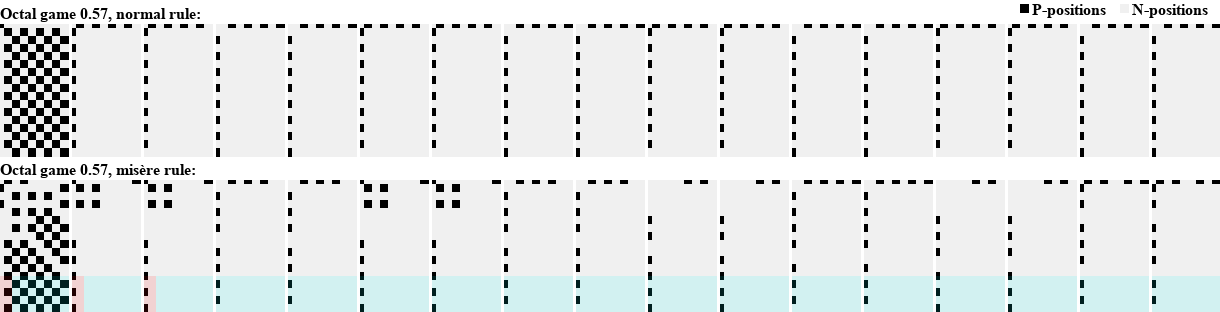}
    \includegraphics[width=0.9\linewidth]{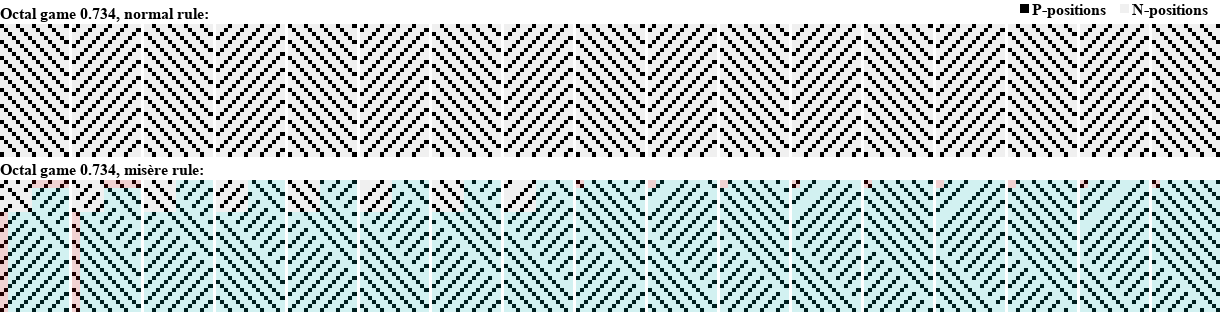}

    \caption{Outcomes of smaller three-heap positions for 0.26, 0.512, 0.57 and 0.734}
    \label{fig:algper_oc_images}
\end{figure}

Taking the simplest of these, $\mathbf{0.73\dot{4}}$, as an example (its image is shown in Figure~\ref{fig:algper_oc_images}), the $\varP$-positions in its three categories are respectively:
\begin{align*}
    \varP_{\mathrm{cor}} &\coloneq \underbrace{(1, 1, \ldots, 1)}_{\text{odd number}}; \\
    \varP_{\mathrm{inn}} &\coloneq \{u \mid u\in\varU,\ \max(u)> 2, \\
            &~~~~~~~~\Delta(u)\leq 1,\ 2\mid \Sigma(u),\ 4\mid (\Sigma(u)+\Omega(u))\}; \\
    \varP_{\mathrm{out}} &\coloneq \{u \mid u\in\varU,\ \max(u)> 2, \\
            &~~~~~~~~\Delta(u)>1,\ 2\mid \Delta(u),\ 4\mid (\Delta(u)+\Omega(u))\}.
\end{align*}
Here $\Sigma(u)$ denotes the sum of all heaps, $\Omega(u)$ denotes the number of odd-sized heaps in $u$, and the vertical bar $\mid$ after the numbers denotes divisibility.
The correctness of the $\varP$-positions thus given can be proved by case analysis, but even for the simplest case $\mathbf{0.73\dot{4}}$ it is quite tedious, and for the other algebraically periodic games it is even more difficult.

The following tables list all algebraically periodic impartial octal games we have found together with their parameters, as well as the periodicity
\footnote{Given as a pair (preperiod, period).}
of outcomes of single-heap positions (starting from size 1, assuming algebraic periodicity holds):

\begin{center}
    \begin{tabular}{|c|cccc||c|cccc|}
        \hline
        Code & $p_1$ & $t_1$ & $p_2$ & $t_2$ & Code & $p_1$ & $t_1$ & $p_2$ & $t_2$ \\
        \hline
        $\mathbf{0.157}$ & 22 & 6 & 1 & 3       & $\mathbf{0.57\dot{4}}$ & 19 & 4 & 1 & 2 \\
        $\mathbf{0.175}$ & 19 & 4 & 1 & 2       & $\mathbf{0.71\dot{6}}$ & 22 & 2 & 2 & 1 \\
        $\mathbf{0.2\dot{6}\bar{0}}$ & 19 & 4 & 1 & 2 & $\mathbf{0.72\bar{4}}$ & 19 & 4 & 1 & 2 \\
        $\mathbf{0.32\dot{6}}$ & 17 & 4 & 1 & 2 & $\mathbf{0.73\dot{2}}$ & 20 & 4 & 0 & 2 \\
        $\mathbf{0.337}$ & 22 & 4 & 0 & 2       & $\mathbf{0.73\dot{4}}$ & 4 & 4 & 0 & 2  \\
        $\mathbf{0.355}$ & 20 & 2 & 1 & 1       & $\mathbf{0.73\dot{6}}$ & 19 & 4 & 0 & 2 \\
        $\mathbf{0.357}$ & 28 & 2 & 1 & 3       & $\mathbf{0.75\dot{4}}$ & 17 & 2 & 1 & 1 \\
        $\mathbf{0.51\dot{2}}$ & 16 & 6 & 3 & 6 & $\mathbf{0.75\dot{6}}$ & 24 & 2 & 1 & 3 \\
        $\mathbf{0.51\dot{6}}$ & 20 & 6 & 1 & 3 & $\mathbf{4.\ddot{0}\dot{6}}$ & 16 & 4 & 1 & 2 \\
        $\mathbf{0.53\dot{4}}$ & 19 & 4 & 1 & 2 & $\mathbf{4.1\dot{6}}$ & 16 & 4 & 1 & 2 \\
        $\mathbf{0.55\dot{2}}$ & 16 & 6 & 1 & 3 & $\mathbf{4.5\dot{2}}$ & 17 & 4 & 1 & 2 \\
        $\mathbf{0.55\dot{6}}$ & 20 & 6 & 1 & 3 & $\mathbf{4.7\dot{0}}$ & 16 & 2 & 1 & 1 \\
        $\mathbf{0.57\dot{0}}$ & 20 & 4 & 1 & 2 & &&&& \\
        \hline
    \end{tabular}
\end{center}

\begin{center}
    \small
    \begin{tabular}{|c|c|c|}
        \hline
        Code & Outcomes & Period \\
        \hline
        $\mathbf{0.157}$          & ${}_{\tabP\tabN\tabN\tabN\tabN\tabN\tabN\tabN\tabN\tabN\tabN\tabN\tabN\tabN\tabN\tabN\tabN\tabN\tabN\tabN\tabN\tabN\tabP\tabP\tabP\tabN\tabN\tabN\tabP\tabP\tabP\tabN\tabN\tabN\tabP \ldots}$ & $(19,6)$    \\
        $\mathbf{0.175}$          & ${}_{\tabP\tabN\tabN\tabN\tabN\tabN\tabN\tabN\tabN\tabN\tabN\tabN\tabN\tabN\tabN\tabN\tabP\tabP\tabN\tabN\tabP\tabP\tabN\tabN\tabP\tabP\tabN\tabN\tabP\tabP\tabN\tabN\tabP\tabP\tabN \ldots}$ & $(14,4)$    \\
        $\mathbf{0.2\dot{6}\bar{0}}$ & ${}_{\tabN\tabN\tabP\tabN\tabN\tabN\tabP\tabN\tabN\tabN\tabN\tabP\tabN\tabN\tabN\tabP\tabN\tabN\tabN\tabP\tabN\tabN\tabN\tabP\tabN\tabN\tabN\tabP\tabN\tabN\tabN\tabP\tabN\tabN\tabN \ldots}$ & $(8,4)$     \\
        $\mathbf{0.32\dot{6}}$    & ${}_{\tabP\tabN\tabP\tabN\tabN\tabN\tabP\tabN\tabN\tabN\tabN\tabN\tabN\tabN\tabN\tabN\tabN\tabP\tabN\tabN\tabN\tabP\tabN\tabN\tabN\tabP\tabN\tabN\tabN\tabP\tabN\tabN\tabN\tabP\tabN \ldots}$ & $(14,4)$    \\
        $\mathbf{0.337}$          & ${}_{\tabP\tabN\tabN\tabN\tabP\tabN\tabN\tabN\tabN\tabN\tabN\tabN\tabN\tabN\tabN\tabP\tabN\tabN\tabN\tabP\tabN\tabN\tabN\tabP\tabN\tabN\tabN\tabP\tabN\tabN\tabN\tabP\tabN\tabN\tabN \ldots}$ & $(12,4)$    \\
        $\mathbf{0.355}$          & ${}_{\tabP\tabN\tabN\tabN\tabN\tabN\tabN\tabN\tabN\tabN\tabN\tabN\tabN\tabN\tabP\tabN\tabP\tabN\tabP\tabN\tabP\tabN\tabP\tabN\tabP\tabN\tabP\tabN\tabP\tabN\tabP\tabN\tabP\tabN\tabP \ldots}$ & $(13,2)$    \\
        $\mathbf{0.357}$          & ${}_{\tabP\tabN\tabN\tabN\tabN\tabN\tabN\tabN\tabN\tabN\tabN\tabN\tabN\tabN\tabP\tabN\tabP\tabN\tabP\tabN\tabP\tabN\tabP\tabN\tabP\tabN\tabP\tabN\tabP\tabN\tabP\tabN\tabP\tabN\tabP \ldots}$ & $(13,2)$    \\
        $\mathbf{0.51\dot{2}}$    & ${}_{\tabP\tabN\tabN\tabN\tabN\tabN\tabN\tabN\tabN\tabN\tabN\tabN\tabN\tabP\tabN\tabN\tabN\tabN\tabN\tabP\tabN\tabP\tabN\tabN\tabN\tabP\tabN\tabP\tabN\tabN\tabN\tabP\tabN\tabP\tabN \ldots}$ & $(16,6)$    \\
        $\mathbf{0.51\dot{6}}$    & ${}_{\tabP\tabN\tabN\tabN\tabN\tabN\tabN\tabN\tabN\tabN\tabN\tabN\tabN\tabN\tabN\tabN\tabN\tabN\tabN\tabN\tabP\tabP\tabP\tabN\tabN\tabN\tabP\tabP\tabP\tabN\tabN\tabN\tabP\tabP\tabP \ldots}$ & $(17,6)$    \\
        $\mathbf{0.53\dot{4}}$    & ${}_{\tabP\tabN\tabN\tabN\tabN\tabN\tabN\tabN\tabN\tabN\tabN\tabN\tabN\tabN\tabN\tabN\tabP\tabP\tabN\tabN\tabP\tabP\tabN\tabN\tabP\tabP\tabN\tabN\tabP\tabP\tabN\tabN\tabP\tabP\tabN \ldots}$ & $(14,4)$    \\
        $\mathbf{0.55\dot{2}}$    & ${}_{\tabP\tabN\tabN\tabN\tabN\tabN\tabN\tabN\tabN\tabN\tabN\tabN\tabN\tabN\tabP\tabN\tabN\tabN\tabN\tabP\tabP\tabP\tabN\tabN\tabN\tabP\tabP\tabP\tabN\tabN\tabN\tabP\tabP\tabP\tabN \ldots}$ & $(16,6)$    \\
        $\mathbf{0.55\dot{6}}$    & ${}_{\tabP\tabN\tabN\tabN\tabN\tabN\tabN\tabN\tabN\tabN\tabN\tabN\tabN\tabN\tabN\tabN\tabN\tabN\tabN\tabN\tabP\tabP\tabP\tabN\tabN\tabN\tabP\tabP\tabP\tabN\tabN\tabN\tabP\tabP\tabP \ldots}$ & $(17,6)$    \\
        $\mathbf{0.57\dot{0}}$    & ${}_{\tabP\tabN\tabN\tabN\tabN\tabN\tabN\tabN\tabN\tabN\tabN\tabN\tabN\tabN\tabN\tabP\tabP\tabN\tabN\tabP\tabP\tabN\tabN\tabP\tabP\tabN\tabN\tabP\tabP\tabN\tabN\tabP\tabP\tabN\tabN \ldots}$ & $(13,4)$    \\
        $\mathbf{0.57\dot{4}}$    & ${}_{\tabP\tabN\tabN\tabN\tabN\tabN\tabN\tabN\tabN\tabN\tabN\tabN\tabN\tabN\tabN\tabN\tabP\tabP\tabN\tabN\tabP\tabP\tabN\tabN\tabP\tabP\tabN\tabN\tabP\tabP\tabN\tabN\tabP\tabP\tabN \ldots}$ & $(14,4)$    \\
        $\mathbf{0.71\dot{6}}$    & ${}_{\tabP\tabN\tabN\tabN\tabP\tabN\tabP\tabN\tabN\tabN\tabN\tabN\tabN\tabN\tabN\tabN\tabN\tabP\tabN\tabN\tabN\tabN\tabN\tabP\tabN\tabP\tabN\tabP\tabN\tabP\tabN\tabP\tabN\tabP\tabN \ldots}$ & $(22,2)$    \\
        $\mathbf{0.72\bar{4}}$    & ${}_{\tabP\tabN\tabP\tabN\tabN\tabN\tabP\tabN\tabN\tabN\tabN\tabN\tabN\tabN\tabN\tabN\tabN\tabP\tabN\tabN\tabN\tabP\tabN\tabN\tabN\tabP\tabN\tabN\tabN\tabP\tabN\tabN\tabN\tabP\tabN \ldots}$ & $(14,4)$    \\
        $\mathbf{0.73\dot{2}}$    & ${}_{\tabP\tabN\tabN\tabN\tabP\tabN\tabN\tabN\tabN\tabN\tabN\tabP\tabN\tabN\tabN\tabP\tabN\tabN\tabN\tabP\tabN\tabN\tabN\tabP\tabN\tabN\tabN\tabP\tabN\tabN\tabN\tabP\tabN\tabN\tabN \ldots}$ & $(8,4)$     \\
        $\mathbf{0.73\dot{4}}$    & ${}_{\tabP\tabN\tabN\tabN\tabN\tabN\tabN\tabP\tabN\tabN\tabN\tabP\tabN\tabN\tabN\tabP\tabN\tabN\tabN\tabP\tabN\tabN\tabN\tabP\tabN\tabN\tabN\tabP\tabN\tabN\tabN\tabP\tabN\tabN\tabN \ldots}$ & $(4,4)$     \\
        $\mathbf{0.73\dot{6}}$    & ${}_{\tabP\tabN\tabN\tabN\tabP\tabN\tabN\tabN\tabN\tabN\tabN\tabN\tabN\tabP\tabN\tabN\tabN\tabP\tabN\tabN\tabN\tabP\tabN\tabN\tabN\tabP\tabN\tabN\tabN\tabP\tabN\tabN\tabN\tabP\tabN \ldots}$ & $(10,4)$    \\
        $\mathbf{0.75\dot{4}}$    & ${}_{\tabP\tabN\tabN\tabN\tabN\tabN\tabN\tabN\tabN\tabN\tabP\tabN\tabP\tabN\tabP\tabN\tabP\tabN\tabP\tabN\tabP\tabN\tabP\tabN\tabP\tabN\tabP\tabN\tabP\tabN\tabP\tabN\tabP\tabN\tabP \ldots}$ & $(9,2)$     \\
        $\mathbf{0.75\dot{6}}$    & ${}_{\tabP\tabN\tabN\tabN\tabN\tabN\tabN\tabN\tabN\tabN\tabP\tabN\tabP\tabN\tabP\tabN\tabP\tabN\tabP\tabN\tabP\tabN\tabP\tabN\tabP\tabN\tabP\tabN\tabP\tabN\tabP\tabN\tabP\tabN\tabP \ldots}$ & $(9,2)$     \\
        $\mathbf{4.\ddot{0}\dot{6}}$ & ${}_{\tabN\tabN\tabP\tabN\tabN\tabN\tabN\tabN\tabN\tabN\tabN\tabP\tabN\tabN\tabN\tabP\tabN\tabN\tabN\tabP\tabN\tabN\tabN\tabP\tabN\tabN\tabN\tabP\tabN\tabN\tabN\tabP\tabN\tabN\tabN \ldots}$ & $(8,4)$     \\
        $\mathbf{4.1\dot{6}}$     & ${}_{\tabP\tabN\tabN\tabN\tabN\tabN\tabN\tabN\tabN\tabN\tabN\tabP\tabN\tabN\tabP\tabP\tabN\tabN\tabP\tabP\tabN\tabN\tabP\tabP\tabN\tabN\tabP\tabP\tabN\tabN\tabP\tabP\tabN\tabN\tabP \ldots}$ & $(11,4)$    \\
        $\mathbf{4.5\dot{2}}$     & ${}_{\tabP\tabN\tabN\tabN\tabN\tabN\tabN\tabN\tabN\tabP\tabN\tabN\tabP\tabP\tabN\tabN\tabP\tabP\tabN\tabN\tabP\tabP\tabN\tabN\tabP\tabP\tabN\tabN\tabP\tabP\tabN\tabN\tabP\tabP\tabN \ldots}$ & $(9,4)$     \\
        $\mathbf{4.7\dot{0}}$     & ${}_{\tabP\tabN\tabN\tabN\tabN\tabN\tabN\tabN\tabP\tabN\tabP\tabN\tabP\tabN\tabP\tabN\tabP\tabN\tabP\tabN\tabP\tabN\tabP\tabN\tabP\tabN\tabP\tabN\tabP\tabN\tabP\tabN\tabP\tabN\tabP \ldots}$ & $(7,2)$     \\
        \hline
    \end{tabular}
    \normalsize
\end{center}

The above results all come from observations of computations on smaller positions; we have not yet found a method for rigorously proving such algebraic periodicity.
To obtain complete solutions to these games, the remaining problem is:
\begin{problem}
    How can one rigorously prove the algebraic periodicity that has been found?
\end{problem}

In fact, we observe that a considerable number of unsolved partizan octal games under misère play also appear to exhibit non-trivial algebraic periodicity (i.e., beyond the case where an outcome DFA exists and ordinary periodicity is present),
but the situation is more complex, and sometimes it is difficult to distinguish whether they have ordinary periodicity or algebraic periodicity; hence we will not discuss this further here.

Moreover, although partizan octal games under normal play do not satisfy the conditions of the Sprague--Grundy theorem, we have not found any among them that exhibit non-trivial algebraic periodicity.

\begin{problem}
    Do normal-play partizan octal games admit no non-trivial algebraic periodicity?
    If so, does this suggest some specific structure of the partizan setting under normal play that distinguishes it from misère play (analogous to the Sprague--Grundy theorem in the impartial case)?
\end{problem}

As for the two exceptional cases \textbf{0.145} and \textbf{0.54},
we observe that their outcome images (Figure~\ref{fig:weak_algper_oc_images}) appear very close to those of algebraically periodic games, but not exactly identical.
In fact, there are a small number of ``interference positions'' that break algebraic periodicity, yet the regions in which these interference positions occur seem to be unbounded, making them difficult to eliminate.

\begin{figure}
    \centering
    \includegraphics[width=0.9\linewidth]{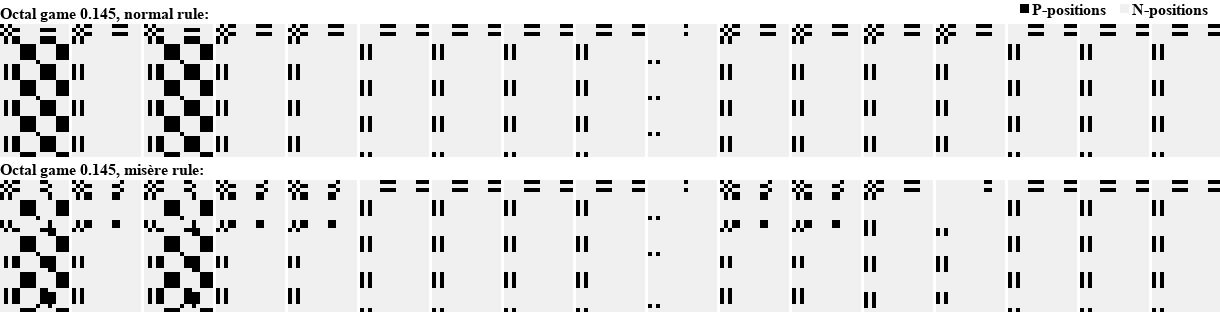}
    \includegraphics[width=0.9\linewidth]{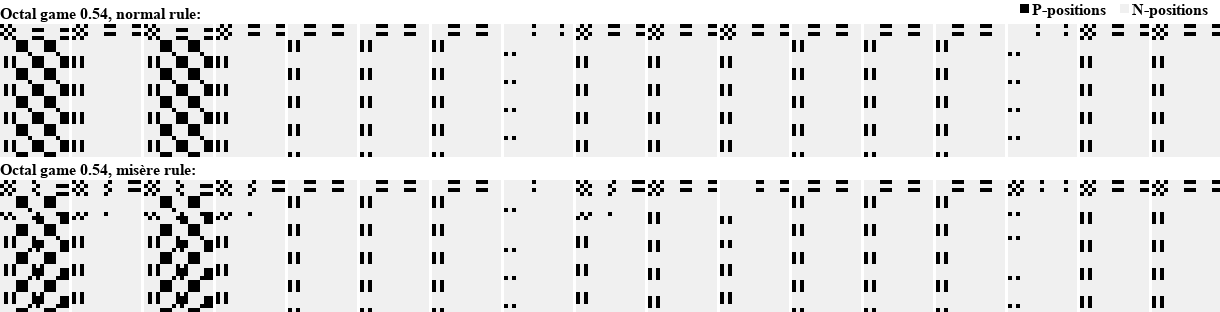}

    \caption{Outcomes of smaller three-heap positions for 0.145 and 0.54}
    \label{fig:weak_algper_oc_images}
\end{figure}

\begin{problem}
    Can one define a notion of ``weak algebraic periodicity'' that encompasses the two exceptional cases $\mathbf{0.145}$ and $\mathbf{0.54}$?
\end{problem}

\section{Kotzig's Nim}
\subsection{Background}

\textit{Kotzig's nim}, also known as \textit{modular nim}, is one of the most classical impartial games, invented by Anton Kotzig in 1946.
The rules of the game are as follows:
Given a circular board with $n$ cells, two players alternate moving a single token clockwise, where the number of steps moved on each turn is taken from a fixed finite set $S$ of positive integers (called the \textit{move set}),
and the token may not be moved to any cell that has already been visited. Under normal play, the player who cannot move loses and the other wins; under misère play, the player who cannot move wins.

Prior research on Kotzig's nim has mainly focused on establishing and proving the periodicity of outcomes and SG values for games with small move sets under normal play:

(1)~\cite{Kotzig1995} gave the outcomes for $S=\{1, 2\}$ and $\{2, 3\}$, as well as the outcomes for most lengths (those not congruent to $3$ or $4$ modulo $7$) for $S=\{3, 4\}$.

(2)~\cite{Kotzig2005} included the outcomes for $S=\{1, 3\}$ due to Nowakowski.

(3)~\cite{Kotzig2014} gave the outcomes for $S=\{1, 4\}$ and proposed a series of conjectures on outcomes.

(4)~\cite{Kotzig2019} gave the SG values for $S=\{1, 2\}$ and $\{1, 3\}$, as well as the SG values for a small subset of lengths (those congruent to $2a$ modulo $2a+1$) for $S=\{a, a+1\}$ with $a>1$.

In addition,~\cite{Kotzig2025} studied the computational complexity of \textit{Geography games}, proving that solving Geography games on bipartite graphs or directed acyclic graphs in the general case is NP-hard.
That paper studied Kotzig's game as a special case of Geography games, and partially resolved a conjecture of Tan and Ward from~\cite{Kotzig2014} concerning a particular move set.

GONC~\cite{GONC} included this game and posed the following questions:

``\textbf{A6.~(17)}~Extend the analysis of KOTZIG'S NIM.

    Is the game eventually periodic in terms of the length of the circle for every finite move set?
    Analyze the misère version of KOTZIG'S NIM.''

In this chapter, we will prove this periodicity conjecture for outcomes — in fact, we generalise it to broader settings including misère play and partizan versions.
We will also provide periodicities of outcomes for some new unresolved cases.

\subsection{Line Games and Their Variants}

\subsubsection{Line Games}

In this chapter, we always denote the move set by $S$, and write $s_{\max}=\max S$ for the maximum step size.

We first study the \textit{line game} mentioned in~\cite{Kotzig2019}.
It is equivalent to Kotzig's nim with the restriction that the token may not move far enough around the board to complete a full circle; in other words, the circular board is cut open at the starting point and straightened into a line:

\begin{center}
    \begin{tikzpicture}[scale=0.5,rotate=90]

        \foreach \deg/\i in {0/$\mathbf{T}$,(1*36)/$c_1$,(2*36)/$c_2$,(3*36)/$c_3$,(4*36)/$c_4$} {
            \filldraw[draw=black,fill=white] ({cos(\deg)*1.5}, {sin(\deg)*1.5}) -- ({cos(\deg)*2.5}, {sin(\deg)*2.5}) arc (\deg:\deg+36:2.5) -- ({cos(\deg+36)*1.5}, {sin(\deg+36)*1.5}) arc (\deg+36:\deg:1.5);
            \node at ({cos(\deg+18)*2}, {sin(\deg+18)*2}) {\i};
        }

        \draw (2.5, 0)--(2.5, -3.5);\draw (1.5, 0)--(1.5, -3.5);
        \node at(2,-0.6) {$c_{n\!-\!1}$};\draw (2.5,-1.2)--(1.5,-1.2);\node at(2,-1.8) {$\ldots$};

        \draw (-1.5, 0)--(-1.5, -3.5);\draw (-2.5, 0)--(-2.5, -3.5);
        \node at(-2,-0.6) {$c_5$};\draw (-2.5,-1.2)--(-1.5,-1.2);\node at(-2,-1.8) {$\ldots$};

        \draw[dashed] (2.5, 0)--(3.5,0);\node at (3.5,0) [above] {starting point};

        \draw ({cos(180)*2.5}, {sin(180)*2.5-3.5}) arc (180:360:2.5);
        \draw ({cos(180)*1.5}, {sin(180)*1.5-3.5}) arc (180:360:1.5);

    \end{tikzpicture}
    \raisebox{0.2cm}{$~~\Rightarrow~~$}
    \begin{tikzpicture}[scale=0.5]
        \draw (0, 0)--(8.5,0)--(8.5,1)--(0,1)--(0,0);

        \node at (0.6,0.5) {$\mathbf{M}$}; \draw (1.2,0)--(1.2,1);\node at (1.8,0.5) {$c_1$}; \draw (2.4,0)--(2.4,1);\node at (3.0,0.5) {$c_2$};
        \draw (3.6,0)--(3.6,1);\node at (4.2,0.5) {$c_3$}; \draw (4.8,0)--(4.8,1);\node at (5.4,0.5) {$\ldots$};

        \draw (7.3,0)--(7.3,1);\node at (7.9,0.5) {$c_{n\!-\!1}$};
    \end{tikzpicture}

\end{center}

We denote the token by $\m$; $\e$ and $\x$ denote an \textit{empty cell} and an \textit{obstacle}, respectively: the token cannot move to a cell occupied by the latter.
We represent a position of a line game of board length $n$ by a word of the form
\[
    w = \m c_1 c_2 \ldots c_{n-1},\quad c_i \in \varSigma \coloneq \{\e, \x\}
\]
(note that the cell occupied by the token is also counted in the board length).
A legal move consists of moving the token $\m$ to the right by $s\in S$ cells, with the condition that the destination cell must be an empty cell $\e$; after the move, the $s$ characters to the left of the token are all deleted:
\[
    \m c_1 c_2 \ldots c_{s-1} \e c_{s+1} \ldots  c_{n-1}  \mapsto \m c_{s+1} \ldots c_{n-1}.
\]
In fact, under the rules of Kotzig's nim, the empty cell $\e$ originally occupied by the token would become an obstacle $\x$;
however, in the line game, all characters to the left of the token play no further role, so they can simply be deleted.
It is easy to see that if the initial position contains no obstacles, the game is equivalent to the subtraction game with move set $S$ (see~\cite{subs_game}); hence the line game can also be viewed as a generalisation of subtraction games with obstacles.

By means of the ``label-flipping'' method mentioned in the previous chapter, we can readily define a partizan version of Kotzig's nim and its labelled positions:
define the label set $\varTheta\coloneq \{\rmL, \rmR\}$,
assign two different move sets $S_\rmL$ and $S_\rmR$ to the two players Left and Right, respectively; labelled positions are of the form
\[
    w = \theta c_1 c_2 \ldots c_{n-1},\quad c_i \in \varSigma,\ \theta \in\varTheta.
\]
A position with label $\theta$ is moved according to the step set $S_\theta$ corresponding to that label, and the label is flipped to the other one after the move.
We denote the resulting game rules by $S_\rmL \col S_\rmR$.

Line games have a nice property: we can determine not only their outcomes but also their SG values using DFAs.
In fact,~\cite{Kotzig2019} already provided the idea for constructing such a DFA.
We give a proof of this, since we will subsequently prove the stronger results~\ref{line_dfa} and~\ref{line_sg_dfa}.

\subsubsection{FA-Line Games}

As can be seen, if we regard labels as states and the remaining characters as alphabet symbols, the move rules of such a line game closely resemble the transitions of an FA with only two states.
We can further generalise this game by imitating the structure of an FA, yielding \textit{FA-line games}: the labels are no longer restricted to two, and the label after a move is determined by the current label and the characters crossed.
Specifically, the structure is a quadruple
\[
    G= \langle \varSigma, \varTheta, \varDelta, \theta_0, \gamma\rangle,
\]
where

(1)~$\varSigma$ is the \textit{alphabet}. In this paper we only consider
\[
    \varSigma = \{\e, \x\}.
\]

(2)~$\varTheta$ is a finite set of \textit{labels}, called the \textit{label set} (analogous to the state set of an FA).
    The token carries one of these labels at any time.
    For brevity, we directly use the label symbol to denote the token.

(3)~$\varDelta$ is a mapping from $\varTheta \times \varSigma^* \e$ to the power set of $\varTheta$, called the \textit{move mapping} (analogous to the transition mapping of an FA),
where a pair $(\theta, v\e)\mapsto \theta'$ means:
if the token carries label $\theta$, and to its right there is a factor $v\e$, it may cross over $v$, land on the empty cell $\e$, and change its label to $\theta'$.

To ensure finiteness, we require that the move mapping has a finite \textit{maximum step size} $s_{\max}$, defined as:
\begin{align*}
    S        &\coloneq \{ |w\e| \mid \exists \theta \in \varTheta,\ \varDelta(\theta,w\e) \neq \varnothing \},\\
    s_{\max} &\coloneq \max S.
\end{align*}

(4)~ $\theta_0 \in\varTheta$ is the \textit{initial label}, i.e., the label carried by the token in the initial position (analogous to the designated initial state of an FA).

(5)~ $\gamma$ is a \textit{seed}
\footnote{This term comes from~\cite{Kotzig2019}; what we define here is in fact a labelled variant thereof.}
, consisting of $s_{\max}$ mappings
\[
    \gamma_i: \varTheta \rightarrow \{\x, \rmP, \rmN\},\quad i=1,2,\ldots,s_{\max},
\]
which artificially assign outcomes to the $s_{\max}$ positions beyond the right end of the board (analogous to the designated final states of an FA).
When the token attempts to enter the $i$-th position among the last $s_{\max}$ cells according to the move mapping and changes its label to $\theta$, the mapping
\[
    \gamma_i(\theta) \in \operatorname{Im}(\gamma) \coloneq \{\x, \e_{\varP}, \e_{\varN}\}
\]
gives the result of this move:
$\x$ means the move is illegal (the destination cell is an obstacle $\x$); $\e_{\varP}$ means the destination cell is empty, but after landing there the moving player immediately loses; $\e_{\varN}$ means the destination cell is empty and the moving player immediately wins.
We require that whether each cell is an obstacle must be consistent across all labels: for every label $\theta\in\varTheta$,
either $\gamma_i(\theta)=\x$ for all $\theta$, or $\gamma_i(\theta)\neq\x$ for all $\theta$.
We define
\begin{align*}
    \bar{\gamma}_i & \coloneq
    \begin{cases}
        \x, & \gamma_i(\theta)=\x,\ \text{i.e., the $i$-th cell is an obstacle }\x; \\
        \e, & \gamma_i(\theta)\neq\x,\ \text{i.e., the $i$-th cell is an empty cell }\e.
    \end{cases}\\
    \bar{\gamma} & \coloneq \bar{\gamma}_1\bar{\gamma}_2\ldots \bar{\gamma}_{s_{\max}}.
\end{align*}
We denote the set of all seeds satisfying these requirements by $\varGamma$.

We can represent a seed as a matrix whose rows correspond to labels and columns to positions:
\[
    \gamma =
    \begin{matrix}
        \theta_1 : \\
        \theta_2 : \\
        \ldots \\
        \theta_{|\varTheta|}: \\
    \end{matrix}
    \begin{bmatrix}
        \gamma_1(\theta_1) & \gamma_2(\theta_1) & \ldots & \gamma_{s_{\max}}(\theta_1) \\
        \gamma_1(\theta_2) & \gamma_2(\theta_2) & \ldots & \gamma_{s_{\max}}(\theta_2) \\
        \vdots & \vdots & & \vdots \\
        \gamma_1(\theta_{|\varTheta|}) & \gamma_2(\theta_{|\varTheta|}) & \ldots & \gamma_{s_{\max}}(\theta_{|\varTheta|}) \\
    \end{bmatrix}
\]
where the $i$-th column is either all $\x$ (in which case $\bar{\gamma}_i=\x$) or none are $\x$ (in which case $\bar{\gamma}_i=\e$).

A position of an FA-line game consists, from left to right, of a labelled token, a sequence of symbols from $\varSigma$, and a seed
\footnote{Although by our definition the seed is fixed for the entire game, we include it as part of the position for intuitive clarity.}:
\begin{align*}
    & \theta c_1 c_2 \ldots c_{n-1} \gamma_1 \gamma_2\ldots\gamma_{s_{\max}}, \\
    & \text{~~~~where~~}\theta \in\varTheta, c_i\in\varSigma, \gamma\in\varGamma.
\end{align*}
The token $\theta$ moves from left to right and changes labels according to the characters crossed, until after some move it enters the seed portion (where the characters in $\bar{\gamma}$ that are crossed are also used for label transitions), at which point the outcome of the game is given by the seed.
That is, the seed effectively extends the board by an additional $s_{\max}$ cells, except that once the token lands on an empty cell within it, the game ends immediately and the winner is determined according to the (changed) label.

It is easy to see that if the seed contains no empty cells, it is equivalent to the original rule without a seed.
On this basis, the original impartial line game can be seen as a special case of FA-line games with
\begin{align*}
    \varTheta &= \{\m\}; \\
    \varDelta &= \{(\m, v\e)\mapsto \m' \mid v\in\varSigma^{s-1},\ s\in S \},
\end{align*}
while the partizan game is the special case with
\begin{align*}
    \varTheta &= \{\rmL, \rmR\}; \\
    \varDelta &= \{(\rmL, v\e)\mapsto \rmR \mid v\in\varSigma^{s-1},\ s\in S_{\rmL} \} \\
              &~~~~ \cup \{(\rmR, v\e)\mapsto \rmL \mid v\in\varSigma^{s-1},\ s\in S_{\rmR} \}.
\end{align*}

\subsubsection{Related DFAs}

Despite the more complex rules, FA-line games retain the property that the outcomes of positions can be determined by a DFA:

\begin{theorem}\label{line_dfa}
    For any FA-line game, any seed $\gamma\in\varGamma$, and any label $\theta_0 \in\varTheta$, the set of character parts of all $\varN$-positions with that seed and label,
    \[
        W = \{w \mid w\in\varSigma^*,\ \gamma w\theta_0 \in \varN\},
    \]
    is a regular language over $\varSigma$. The same holds for $\varP$-positions.
\end{theorem}

\begin{proof}
    Unlike the usual direction, the DFA constructed in this proof reads characters from right to left.
    To obtain an equivalent left-to-right DFA, one simply reverses all state transitions and then determinises.

    We take the set of all seeds $\varGamma$ as the state set of the DFA, and construct the transition mapping according to the recursive outcome relations,
    so that starting from the initial state $\gamma$, it reads characters one by one from right to left, and after all characters are read, the outcome is given according to the token label.
    The key point that guarantees finiteness of the state set is that the legal move lengths are bounded:
    since a seed $\gamma$ contains $s_{\max}$ positions (the length of the seed), after a new position is recorded by $\gamma$, the number of positions becomes $s_{\max}+1$, but the rightmost position is unreachable by any legal move and can therefore be deleted, keeping the number of positions bounded by $s_{\max}$.

    If the position is $\theta_0\gamma$, i.e., there are no characters between the token and the seed, then any legal move of the token enters the seed directly, and the outcome can be computed immediately:
    \[
        o(\theta_0\gamma) =
        \begin{cases}
            \varN, & \exists s\in S,\ \gamma_s(\varDelta(\bar{\gamma}_{1}\ldots\bar{\gamma}_{s-1}\e ,\theta_0)) = \rmN; \\
            \varP, & \text{otherwise}.
        \end{cases}
    \]
    This gives the final states of the DFA: $\gamma$ is a final state if and only if $\theta_0\gamma$ has the corresponding outcome.

    For the transition mapping, if an obstacle $\x$ is encountered, we simply delete the rightmost column and add a column of obstacles $\x$ on the left:
    \[
        \x\cdot
        \begin{bmatrix}
            \gamma_1(\theta_1) & \ldots & \gamma_{s_{\max}}(\theta_1) \\
            \gamma_1(\theta_2) & \ldots & \gamma_{s_{\max}}(\theta_2) \\
            \vdots & & \vdots \\
            \gamma_1(\theta_{|\varTheta|}) & \ldots & \gamma_{s_{\max}}(\theta_{|\varTheta|}) \\
        \end{bmatrix}
        =
        \begin{bmatrix}
            \x & \gamma_1(\theta_1) & \ldots & \gamma_{s_{\max}-1}(\theta_1) \\
            \x & \gamma_1(\theta_2) & \ldots & \gamma_{s_{\max}-1}(\theta_2) \\
            \vdots & \vdots & & \vdots \\
            \x & \gamma_1(\theta_{|\varTheta|}) & \ldots & \gamma_{s_{\max}-1}(\theta_{|\varTheta|}) \\
        \end{bmatrix}.
    \]

    If an empty cell $\e$ is encountered, we delete the rightmost column and add a column on the left whose entries are computed according to the legal moves for each label:
    \[
        \e\cdot
        \begin{bmatrix}
            \gamma_1(\theta_1) & \ldots & \gamma_{s_{\max}}(\theta_1) \\
            \gamma_1(\theta_2) & \ldots & \gamma_{s_{\max}}(\theta_2) \\
            \vdots & & \vdots \\
            \gamma_1(\theta_{|\varTheta|}) & \ldots & \gamma_{s_{\max}}(\theta_{|\varTheta|}) \\
        \end{bmatrix}
        =
        \begin{bmatrix}
            \e_{o(\theta_1 \gamma)} & \gamma_1(\theta_1) & \ldots & \gamma_{s_{\max}-1}(\theta_1) \\
            \e_{o(\theta_2 \gamma)} & \gamma_1(\theta_2) & \ldots & \gamma_{s_{\max}-1}(\theta_2) \\
            \vdots & \vdots & & \vdots \\
            \e_{o(\theta_{|\varTheta|} \gamma)} & \gamma_1(\theta_{|\varTheta|}) & \ldots & \gamma_{s_{\max}-1}(\theta_{|\varTheta|}) \\
        \end{bmatrix}.
    \]

    This gives the complete structure of a DFA satisfying the requirements.
\end{proof}

Although this proof is constructive, the resulting outcome DFA has an enormous number of states, namely the total number of possible seeds:
\[
    |\varGamma|=(2^{|\varTheta|}+1)^{s_{\max}}.
\]
This grows super-exponentially with the number of labels $|\varTheta|$.
In fact, games always have outcome DFAs with far fewer states; hence we do not use this construction to build DFAs, but only employ this result for subsequent proofs.

For misère play, the same conclusion holds by replacing the outcome $o(\cdot)$ in the above proof with $o^-(\cdot)$.
Moreover, if we replace it with $\varG(\cdot)$, we can also give NFAs (which can be called \textit{SG-value DFAs}) for determining the positions of an FA-line game with any given SG value:
the number of legal moves from any position of Kotzig's nim is always bounded by
\[
    m_{\max} \coloneq
    \begin{cases}
        |S|,                            & \text{impartial}; \\
        \max\{|S_{\rmL}|, |S_{\rmR}|\}, & \text{partizan}.
    \end{cases}
\]
By definition, the $\mex$ of a set of at most $m_{\max}$ numbers cannot exceed $m_{\max}$, so the SG value of any position $u\in\varU$ is also bounded:
\[
    \varG(u) \leq m_{\max}.
\]
We need only change the codomain of the seed $\gamma$ to
\[
    \{\x, \e_0, \e_1, \ldots, \e_{m_{\max}}\}
\]
to obtain an SG-value version of the seed.
Denoting the set of all such seeds by $\varGamma_{\varG}$, we can prove, analogously to~\ref{line_dfa}:

\begin{theorem}\label{line_sg_dfa}
    For any FA-line game, any SG-value seed $\gamma^-\in\varGamma_{\varG}$, any label $\theta_0 \in\varTheta$,
    and any SG value $g\in\mathbb{N}$, the set of character parts of positions with that seed and label whose SG value is $g$,
    \[
        W = \{w \mid w\in\varSigma^*,\ \varG(\gamma^- w\theta_0)=g\},
    \]
    satisfies:

    (1)~When $g \leq m_{\max}$, $W$ is a regular language over $\varSigma$.

    (2)~When $g > m_{\max}$, $W=\varnothing$.
\end{theorem}

Since positions are themselves words, we can naturally define a right action of $\varSigma^*$ on all positions of an FA-line game as concatenation of words (keeping the seed and label unchanged):
\[
    (\theta w \gamma^-) \cdot c \coloneq (\theta wc \gamma^-),\quad \theta\in\varTheta,\ w\in\varSigma^*.
\]
Then we can define a generalized misère quotient under this right action.
Unlike before, since the SG values are bounded, we can also define an equivalence relation for SG values:
for any positions $u_1, u_2\in\varU$,
\[
    u_1 \cong^{\varG} u_2 \\iff \forall v\in\varSigma^*,\ \varG(u_1\cdot v) = \varG(u_2\cdot v),
\]
which yields the \textit{generalized misère quotient for SG values}
\[
    \overline{\varU} \coloneq \varU / \cong^{\varG}.
\]

\begin{proposition}\label{line_finite_misquot}
    For any FA-line game (under normal or misère play), the generalized misère quotient and the generalized misère quotient for SG values defined above are both finite.
\end{proposition}

\begin{proof}
    Analogous to Proposition~\ref{finite_misquot}, it suffices to note that the final states of the outcome DFA of an FA-line game are in bijection with the equivalence classes of the generalized misère quotient;
    similarly, the final states of its SG-value DFA correspond to the equivalence classes of the generalized misère quotient for SG values.
\end{proof}

\begin{proposition}
    For any FA-line game (under normal or misère play), the outcome or SG value of the empty board is periodic.
\end{proposition}

\begin{proof}
    The proof is similar in principle to~\ref{normal_per} (but simpler): it suffices to note that the finiteness of the state set of the outcome DFA makes the sequence of states
    \[
        q_0\cdot\epsilon,\ q_0\cdot\e,\ q_0\cdot\e\e,\ q_0\cdot\e\e\e,\ \ldots,\ q_0\cdot\e^n,\ \ldots
    \]
    in the outcome DFA ultimately periodic, which yields periodic outcomes. The same argument applies to the SG-value DFA.
\end{proof}

\subsection{Kotzig's Nim}

Similar to line games, we can represent a position of Kotzig's nim of board length $n+1$ by a word of the form
\[
    w = \m c_1 c_2 \ldots c_{n-1},\quad c_i \in \varSigma.
\]
When the move set is $S$, a legal move consists of moving the token $\m$ to the right by $s\in S$ cells, with the destination cell required to be empty $\e$; after the move, the cell originally occupied by the token becomes an obstacle $\x$.

When the board length $n > s_{\max}$, a legal move has the form
\[
     \m c_1 c_2 \ldots c_{s-1} \e c_{s+1} \ldots  c_{n-1}  \mapsto \x c_1 c_2 \ldots c_{s-1}\m c_{s+1} \ldots c_{n-1}.
\]
Since the board is circular, we cyclically rotate all characters to the left of $\m$ to the right end, so that $\m$ is always the first character:
\[
    \x c_1 c_2 \ldots c_{s-1}\m c_{s+1} \ldots c_{n-1} \equiv \m c_{s+1} \ldots c_{n-1}\x c_1 c_2 \ldots c_{s-1}.
\]

However, when $n\leq s_{\max}$ (the board length does not exceed the maximum step size), some difficult cases arise; hence in the following we will only discuss the sufficiently long case $n> s_{\max}$.
We denote the set of all such sufficiently long positions by $\varU^>$ (with the token as the first character by convention), and the sets of $\varP$-positions, $\varN$-positions and terminal positions within it by $\varP^>, \varN^>, \varT^>$, respectively.

It is easy to see that $\varU^>$ is closed under the right action of $\varSigma^*$, so the (SG-value) generalized misère quotient we defined can be restricted to $\varU^>$, yielding $\overline{\varU^>}$ and its subsets $\overline{\varP^>}, \overline{\varN^>}$.

We can generalise Kotzig's nim to \textit{FA-Kotzig's nim} in a manner similar to the previous section; the only point requiring discussion is the placement of the seed.
For a linear board, it suffices to place the seed at the end of the board.
For a circular board such as Kotzig's nim, we need to introduce an additional parameter, the \textit{number of laps} $r$:
the token may circle the board at most $r-1$ times, after which the seed appears in front of the starting position, and the token enters the seed when it attempts to cross the starting position to complete the $r$-th lap.
As illustrated below:

\begin{center}
    \begin{tikzpicture}[scale=0.7,rotate=90]

        \filldraw[draw=black,fill=gray!20] (2.5, 0)--(2.5, 6)--(1.5, 6)--(1.5, 0);
        \node at(2.1,0.6) {$\gamma_1$};\draw (2.5,1.2)--(1.5,1.2);\node at(2.1,1.8) {$\gamma_2$};\draw (2.5,2.4)--(1.5,2.4);\node at(2.1,3.0) {$\ldots$};
        \node at(2.1,5.4) {$\gamma_{s_{\max}}$};\draw (2.5,4.8)--(1.5,4.8);

        \foreach \deg/\i in {0/$\mathbf{T}$,(1*36)/$c_1$,(2*36)/$c_2$,(3*36)/$c_3$,(4*36)/$c_4$} {
            \filldraw[draw=black,fill=white] ({cos(\deg)*1.5}, {sin(\deg)*1.5}) -- ({cos(\deg)*2.5}, {sin(\deg)*2.5}) arc (\deg:\deg+36:2.5) -- ({cos(\deg+36)*1.5}, {sin(\deg+36)*1.5}) arc (\deg+36:\deg:1.5);
            \node at ({cos(\deg+18)*2}, {sin(\deg+18)*2}) {\i};
        }

        \draw (2.5, 0)--(2.5, -3.5);\draw (1.5, 0)--(1.5, -3.5);
        \node at(2,-0.6) {$c_{n\!-\!1}$};\draw (2.5,-1.2)--(1.5,-1.2);\node at(2,-1.8) {$\ldots$};

        \draw (-1.5, 0)--(-1.5, -3.5);\draw (-2.5, 0)--(-2.5, -3.5);
        \node at(-2,-0.6) {$c_5$};\draw (-2.5,-1.2)--(-1.5,-1.2);\node at(-2,-1.8) {$\ldots$};

        \draw[dashed] (2.5, 0)--(4.5,0);\node at (4.5,0) [above] {starting point};
        \draw[->] (4,-1.5)--(4,0);\draw[->] (4,0)--(4,1.5);\node at (4,1.5) [left] {completed $r-1$ laps};
        \draw[->] (4,0) arc (0:36:2.5); \node at (3.2,1.5) [left] {not yet completed};

        \draw ({cos(180)*2.5}, {sin(180)*2.5-3.5}) arc (180:360:2.5);
        \draw ({cos(180)*1.5}, {sin(180)*1.5-3.5}) arc (180:360:1.5);

    \end{tikzpicture}
\end{center}

We now extend the conclusions of the previous section from FA-line games (which correspond to FA-Kotzig's nim with zero laps) to FA-Kotzig's nim. This requires two important facts:
first, even without a lap restriction, the token can make only finitely many laps around the board; second, if all FA-Kotzig's nim games with $r$ laps have finite misère quotients,
then for any FA-Kotzig's nim with $r+1$ laps, we can construct an equivalent FA-line game, thereby showing that it also has a finite misère quotient.

The former is quite intuitive:
\begin{proposition}\label{cyc_finite}
    For FA-Kotzig's nim with maximum step size $s_{\max}$, the token can make at most $s_{\max}-1$ full laps on any board of length $n > s_{\max}$ (i.e., it cannot complete $s_{\max}$ full laps).
\end{proposition}

\begin{proof}
    Since each cell can be visited at most once, the token can make at most $n-1$ moves; each move has maximum length $s_{\max}$, so the total distance travelled is at most $(n-1)s_{\max}$.
    When $n> s_{\max}$, it can travel at most
    \[
        \left\lfloor \frac{(n-1)s_{\max}}{n} \right\rfloor = s_{\max}-\left\lfloor \frac{s_{\max}}{n} \right\rfloor = s_{\max}
    \]
    full laps.
\end{proof}

For the latter, we construct an equivalent FA-line game according to the following principle:
use an FA-line game to simulate the first lap, during which the token records, using its label, all characters of the cells it has passed (including the obstacles it leaves behind);
when the token approaches the initial position, the recorded characters are ``unrolled'' to the right of the initial position to simulate the remaining $r$ laps of the token.
It suffices to take the outcomes of all positions obtained by such ``unrolling'' as the seed of the FA-line game; then the token need not continue moving to yield the same outcome.
The recorded sequence of characters has infinitely many possible arrangements, but as long as the misère quotient is finite, only its equivalence class needs to be recorded to achieve the same effect, hence only finitely many labels are needed.

\begin{proposition}\label{cyc_reduce}
Fix a label set $\Theta$, a move mapping $\Delta$, and a seed $\gamma$.
Suppose that for every initial label $\theta_0 \in \Theta$, FA-Kotzig's nim with $r$ laps has an outcome DFA (or SG-value DFA)
\[
\mathcal{A}_r^{\theta_0} = \langle Q_r, \Sigma, \delta_r, q_0^{\theta_0}, F_r \rangle,
\]
then FA-Kotzig's nim with $r+1$ laps also has an outcome DFA (or SG-value DFA).
\end{proposition}

\begin{proof}
For any string $w \in \Sigma^*$ of visited cells, define its \textit{signature} as
\[
\sigma(w) \coloneqq \left( \operatorname{suffix}_{s_{\max}-1}(w),\;
\left( \delta_r(q_0^{\theta'}, w) \right)_{\theta' \in \Theta} \right),
\]
where $\delta_r(q_0^{\theta'}, w)$ is the state reached by the DFA $\mathcal{A}_r^{\theta'}$ after reading $w$ from its initial state $q_0^{\theta'}$.

Since $Q_r$ and $\Theta$ are both finite, the set
\[
\Sigma_{\sigma} \coloneqq \{ \sigma(w) \mid w \in \Sigma^* \}
\]
is finite.

We now construct a DFA $\mathcal{A}_{r+1}$ for the $(r+1)$-lap game. Its state set is
\[
Q_{r+1} \coloneqq \Theta \times \Sigma_{\sigma}.
\]
Its initial state is
\[
q_0^{r+1} \coloneqq (\theta_0, \sigma(\epsilon)).
\]

For each state $(\theta, \sigma(w)) \in Q_{r+1}$ and each symbol $c \in \Sigma$, the transition is defined as follows. First, let $\theta'$ be the label obtained from $\theta$ after a move that crosses character $c$ according to the move mapping $\Delta$; if no such move is legal, then $\theta'$ is undefined and the transition goes to a sink state. Otherwise, define
\[
\delta_{r+1}((\theta, \sigma(w)), c) \coloneqq (\theta', \sigma(wc)).
\]
The signature $\sigma(wc)$ is computed explicitly as
\[
\sigma(wc) = \left( \operatorname{suffix}_{s_{\max}-1}(wc),\;
\left( \delta_r(\delta_r(q_0^{\theta'}, w), c) \right)_{\theta' \in \Theta} \right),
\]
using the transition function $\delta_r$ of the $r$-lap DFAs.

It remains to specify the set of final states $F_{r+1} \subseteq Q_{r+1}$. For a state $(\theta, \sigma(w))$, this state corresponds to a position in which the token (carrying label $\theta$) is about to enter the seed region, with $w$ being the string of cells already visited. The outcome of such a position in the $(r+1)$-lap game is determined as follows: when the token enters the seed, the game effectively continues in the $r$-lap game from the unfolded position represented by $w$. By the inductive hypothesis, this outcome is precisely given by the state $\delta_r(q_0^{\theta}, w)$ in $\mathcal{A}_r^{\theta}$. Hence we define
\[
(\theta, \sigma(w)) \in F_{r+1}
\iff
\delta_r(q_0^{\theta}, w) \in F_r.
\]

The correctness of this construction follows by induction on the number of characters read: after reading any prefix $v \in \Sigma^*$, the DFA $\mathcal{A}_{r+1}$ is in state $(\theta_v, \sigma(w_v))$, where $\theta_v$ is the current label and $w_v$ is the string of cells visited so far. The signature $\sigma(w_v)$ exactly records all information needed to determine, for any possible continuation, the outcome of the game once the token completes its first lap and enters the $r$-lap phase. Therefore, for every position $(\theta_0, w\gamma)$ of board length $n > s_{\max}$, we have
\[
(\theta_0, w\gamma) \in \mathcal{N}
\iff
\delta_{r+1}(q_0^{r+1}, w) \in F_{r+1}.
\]

Thus $\mathcal{A}_{r+1}$ is an outcome DFA for the $(r+1)$-lap game.
Since the remaining positions of length $n \le s_{\max}$ are finite in number, they can be handled separately by the finiteness of regular languages (Theorem 2.3), and the union with the above DFA is again regular by Theorem 2.4. The proof is complete.
\end{proof}

Thus we resolve the conjecture about this game from GONC~\cite{GONC}, in a form that is even far more general than the original conjecture:

\begin{theorem}
    Every FA-Kotzig's nim (under both normal and misère play) has an outcome DFA (or SG-value DFA).
    This includes all impartial and partizan classical Kotzig's nim games.
\end{theorem}


\begin{proof}[Proof of Theorem 4.3]
By Proposition~\ref{cyc_finite}, the token can make at most $s_{\max}-1$ full laps on any board of length $n > s_{\max}$.
We prove by induction on the number of laps $r$ that every $r$-lap FA-Kotzig's nim has an outcome DFA (or SG-value DFA).

For the base case $r = 0$, the game is precisely an FA-line game, which has an outcome DFA by Proposition~\ref{line_dfa}, and an SG-value DFA by Proposition~\ref{line_sg_dfa}.

For the inductive step, assume that every $r$-lap FA-Kotzig's nim (for all initial labels and seeds) has an outcome DFA.
Then by Proposition~\ref{cyc_reduce}, every $(r+1)$-lap FA-Kotzig's nim also has an outcome DFA.

Applying the induction up to $r = s_{\max}-1$, we conclude that every FA-Kotzig's nim with the original lap restriction (i.e., with no explicit lap bound, but with at most $s_{\max}-1$ laps possible) has an outcome DFA.

For positions of length $n \le s_{\max}$, there are only finitely many such positions up to the natural symmetries of the game; hence the sets of positions with each outcome are finite languages, and therefore regular by Theorem~\ref{reg_closure}.
Taking the union of the DFA for long positions and the finite automata for short positions gives a DFA for all positions.
The same argument applies under mis\`ere play and for SG-value DFAs.

Therefore every FA-Kotzig's nim, including all impartial and partizan classical Kotzig's nim games, has an outcome DFA (and, in the impartial case, SG-value DFAs) under both normal and mis\`ere play.
\end{proof}

\subsection{Algorithmic Principles}

\subsubsection{General Principles}

The algorithm for searching for outcome DFAs for Kotzig's nim is essentially the same as the algorithm for octal games described in the previous chapter; we only need to point out the differences:

(1)~For simplicity, we only consider positions with board length greater than $s_{\max}$ (outcomes of shorter positions are easy to compute). For this purpose, the initial state queue contains all states of length $s_{\max}$.

(2)~Legal positions no longer have the restriction that ``the first and last characters must be $\x$''; hence the words in the initial queue are no longer required to have $\x$ as the second character, and words corresponding to final states of the DFA are no longer required to end with $\x$.

More crucially, to prove the correctness of the resulting outcome DFA, the game rules must be expressed in terms of regular language operations (i.e., FA operations).
For impartial Kotzig's nim, $\nxt(\cdot)$ can be implemented by first deleting a prefix (using left quotient as defined in~\ref{reg_closure}), then adding prefixes and suffixes; $\prv(\cdot)$ is obtained by reversing the order of operations.
For a move set $S$, for any set of positions $\varW \subseteq \varTheta\varSigma^*$,
\begin{align*}
    \nxt(\varW) &= \bigcup_{s\in S} \bigcup_{v\in\varSigma^{s-1}} \rmR (\rmL v\e \setminus \varW) \x v;\\
    \prv(\varW) &= \bigcup_{s\in S} \bigcup_{v\in\varSigma^{s-1}} \rmR v\e (\rmL \setminus (\varW ~/~ \x v)).
\end{align*}
For partizan games, the above operations are performed for both labels while flipping the label, and then the union is taken. For Left and Right move sets $S_{\rmL}, S_{\rmR}$,
\begin{align*}
    \nxt(\varW) &= \left(\bigcup_{s\in S_{\rmL}} \bigcup_{v\in\varSigma^{s-1}} \rmR (\rmL v\e \setminus \varW) \x v\right) \\
            &~~~~ \cup \left(\bigcup_{s\in S_{\rmR}} \bigcup_{v\in\varSigma^{s-1}} \rmL (\rmR v\e \setminus \varW) \x v\right);\\
    \prv(\varW) &= \left(\bigcup_{s\in S_{\rmL}} \bigcup_{v\in\varSigma^{s-1}} \rmR v\e (\rmL \setminus (\varW ~/~ \x v))\right) \\
            &~~~~ \cup \left(\bigcup_{s\in S_{\rmR}} \bigcup_{v\in\varSigma^{s-1}} \rmL v\e (\rmL \setminus (\varW ~/~ \x v))\right).
\end{align*}
By Theorem~\ref{reg_closure}, all operations involved are closed for regular languages.

In addition, we need to determine the FA for the set of all terminal positions of the game.
When the board length $n\geq s_{\max}$, the game reaches a terminal position if and only if all next positions under legal step sizes are obstacles $\x$; this can be decided by inspecting the first $s_{\max}$ characters ahead of the token.
For a move set $S$, the set of all terminal positions is
\[
    \varT = \{\m c_1 c_2\ldots c_{n-1} w \mid w\in\varSigma^+,\ c_i
    \begin{cases}
        = \x, & i\in S; \\
        \in \varSigma, & i\notin S,
    \end{cases}\},
\]
while for the partizan version it is
\[
    \varT = \{\theta c_1 c_2\ldots c_{n-1} w \mid \theta\in\{\rmL, \rmR\},\ w\in\varSigma^+,\ c_i
    \begin{cases}
        = \x, & i\in S_{\theta}; \\
        \in \varSigma, & i\notin S_{\theta}.
    \end{cases}\}
\]
All of these are regular languages.

\subsubsection{Restriction to Quasi-Reachable Positions}

The proof in the previous section actually constructed outcome DFAs, but they contain far too many redundant states, and the number of states grows super-exponentially with $s_{\max}$.
Moreover, from the actual computational results, even without redundant states, the growth rate of the number of states is very fast.
However, we are most concerned with games starting from the empty board, which only requires discussing all \textit{reachable} positions from the empty board.
Therefore, we may ignore some unimportant positions and only use DFAs to determine the outcomes of a class of positions that covers all reachable positions.

We call a position $u\in\varTheta\varSigma^*$ \textit{quasi-reachable} if every run of consecutive $\e$'s following any $\x$ has length less than $s_{\max}$.
The set of all quasi-reachable positions can be expressed as
\[
    \varU_{\mathrm{qr}} \coloneq \varTheta \e^* \left( \bigcup_{0< s\leq s_{\max}} \x \e^{s-1} \right)^*.
\]

Clearly any non-quasi-reachable position is unreachable, because a non-quasi-reachable position either contains
\[
    \x \e^{s-1} \x,\quad s> s_{\max},
\]
or ends with
\[
    \x \e^{s-1},\quad s> s_{\max},
\]
both of which imply that the token has made at least one move of length $s>s_{\max}$, violating the game rules.
In fact, the quasi-reachability defined here corresponds to considering only cases that are unreachable due to exceeding the maximum step size.
For example, when $S=\{1, 3\}$, the positions
\[
    \m \e\e\e\e\e\e \x\e\e\x\underbrace{\e}_1\x,\qquad \m \e\e\e\e\e\e\x\underbrace{\e\e\e}_3 \x\x,
\]
are both unreachable, since the gaps of 1 and 3 empty cells respectively can only be produced by moves of lengths 2 and 4, neither of which is in $S$.
But the first position is quasi-reachable (2 is illegal but does not exceed the maximum step size), while the second is not.

Our algorithm for proving correctness of the outcome DFA still works under the quasi-reachability restriction, for the following two reasons:
First, if a position $u\in\varU$ is quasi-reachable, then all positions in $\nxt(u)$ are also quasi-reachable. This is easy to see from the definition of quasi-reachability and the game rules.
Second, from the expression we gave for $\varU_{\mathrm{qr}}$, it is clearly a regular language.
Likewise, the set obtained by removing the finitely many insufficiently long positions,
\[
    \varU_{\mathrm{qr}}^> \coloneq \varU_{\mathrm{qr}} \cap \varU^>,
\]
is of course also regular.

To construct outcome DFAs for quasi-reachable positions, we need to make the following modifications to the algorithm:

(1)~Clearly every prefix of a quasi-reachable position is quasi-reachable. Therefore, if a word to be added to the queue is not quasi-reachable, it need not be enqueued.

(2)~To shield the algorithm from the complexity of non-quasi-reachable positions, we modify the game rules: any non-quasi-reachable position has no legal moves.
Thus its outcome is uniformly $\varP$ (or $\varN$ under misère play).

(3)~The outcome DFA directly obtained in this way may still contain some non-quasi-reachable positions.
To filter these out, we must intersect it with the set of all quasi-reachable positions $\varU_{\mathrm{qr}}$; otherwise it would affect subsequent verification.

\subsubsection{Verification of SG-Value DFAs}

For impartial Kotzig's nim under normal play, we can also use the same algorithm to search for SG-value DFAs.
The SG value is meaningful because it gives the outcome of the disjunctive sum of two or more such games (equivalent to choosing one token on either of two boards and moving it).

In contrast, verifying the obtained SG-value DFA is more involved.
We need to rewrite the definition of SG values in Definition~\ref{def:SG} in set-theoretic form, analogously to Corollary~\ref{outcome_cond}:

\begin{corollary}\label{sg_cond}
    For any sets of positions
    \[
        \varW_{0}, \varW_{1}, \ldots, \varW_{m_{\max}}\subseteq \varU
    \]
    of an impartial game under normal play,
    they are exactly the sets of all positions with SG values $0, 1, \ldots, m_{\max}$ if and only if all of the following conditions hold:

    (1)~$\varT \subseteq \varW_{0}$.

    That is, under normal play, terminal positions always have SG value 0.

    (2)~For any SG value $g \leq m_{\max}$, $\nxt(\varW_{g}) \subseteq \bigcup_{g'\neq g} \varW_{g'}$.

    That is, the SG value of any successor cannot be equal to the original SG value.

    (3)~For any two SG values $g'<g \leq m_{\max}$, $\varW_{g} \subseteq \prv(\varW_{g'})$.

    That is, the SG values of successors can cover every smaller SG value.
\end{corollary}

All of these conditions can also be checked using FA operations.
Although this may involve many DFAs ($m_{\max}+1$ DFAs in total), they differ only in their final states; all other structure is shared.
If the number of states still overflows, we may also restrict the positions to quasi-reachable ones as before; the resulting SG-value DFA is then sufficient to determine the SG values of all positions reachable from the empty board.

\subsection{Computational Results}

\subsubsection{Results for $\{1, 2\}$}

If one requires DFAs that determine the outcomes of all positions, our algorithm can only produce outcome DFAs for the case $S=\{1,2\}$ under both normal and misère play.
The former already has sufficiently detailed results in the literature~\cite{Kotzig1995, Kotzig2019}, so here we will focus on the misère-play case.

The outcome DFA we obtained has 52 states, which can be classified as follows:
\tiny
\begin{align*}
    \varU^{\leq}: & \epsilon, \m, \m\x, \m\e; \\
    \varP^>:      & \m\x\e, \m\e\x, \m\e\x\x, \m\e^3, \m\x\e^3, \m\e\x\e^2, \m\e^2\x\e, \m\e^4\x, \m\e^3\x\e\x, \m\e^3\x\e^2, \\
                  & \m\e^6, \m\e^2\x\e\x\e^2, \m\e^3\x\e\x\e, \m\e^3\x\e^3, \m\e^7, \m\e^6\x\e, \m\e^6\x\e\x, \m\e^6\x\e^2, \m\e^9, \m\e^4\x\e\x\e^3, \\
                  & \m\e^6\x\e\x\e, \m\e^6\x\e^3, \m\e^6\x\e\x\e^2;\\
    \varN^>:      & \m\x\x, \m\e^2, \m\x\e^2, \m\e\x\e, \m\e^2\x, \m\x\e^2\x, \m\e^3\x, \m\e^4, \m\e\x\e^3, \m\e^2\x\e\x, \\
                  & \m\e^2\x\e^2, \m\e^3\x\e, \m\e^5, \m\e^2\x\e\x\e, \m\e^4\x\e, \m\e^4\x\e\x, \m\e^4\x\e^2, \m\e^6\x, \m\e^3\x\e\x\e^2, \m\e^4\x\e\x\e, \\
                  & \m\e^4\x\e^3, \m\e^8, \m\e^3\x\e\x\e^3, \m\e^4\x\e\x\e^2, \m\e^6\x\e\x\e^3.
\end{align*}\normalsize
Note that the outcome DFA presented here only determines the outcomes of sufficiently long positions ($n> s_{\max}=2$);
hence the first 4 states correspond to insufficiently long positions, while the remaining states are in bijection with the classes of the generalized misère quotient $\overline{\varU^>}$,
consisting of 23 $\varP$-classes and 25 $\varN$-classes.

The equivalence relation on positions is still generated by all transitions (103 in total). The non-trivial transitions (53 in total) are:
\tiny
\begin{align*}
    & \m\x^2\cdot\x=\m\x^2,~~ \m\x^2\cdot\e=\m\x^2,~~ \m\x\e\cdot\x=\m\e\x,~~ \m\e\x\cdot\x=\m\e\x^2,~~ \\
    & \m\e\x^2\cdot\x=\m\e\x^2,~~ \m\e\x^2\cdot\e=\m\e\x^2,~~ \m\e\x\e\cdot\x=\m\x\e^2\x,~~ \m\e^2\x\cdot\x=\m\x^2,~~\\
    & \m\x\e^2\x\cdot\x=\m\x^2,~~ \m\x\e^2\x\cdot\e=\m\x\e,~~ \m\x\e^3\cdot\x=\m\e\x^2,~~ \m\x\e^3\cdot\e=\m\x\e,~~ \\
    & \m\e\x\e^2\cdot\x=\m\e\x,~~ \m\e^3\x\cdot\x=\m\x^2,~~ \m\e\x\e^3\cdot\x=\m\x^2,~~ \m\e\x\e^3\cdot\e=\m\e\x\e,~~ \\
    & \m\e^2\x\e\x\cdot\x=\m\x^2,~~ \m\e^2\x\e^2\cdot\x=\m\e^2\x,~~ \m\e^2\x\e^2\cdot\e=\m\e^2\x,~~ \m\e^4\x\cdot\x=\m\e\x^2,~~ \\
    & \m\e^5\cdot\x=\m\x^2,~~ \m\e^2\x\e\x\e\cdot\x=\m\e^2\x,~~ \m\e^3\x\e\x\cdot\x=\m\e\x^2,~~ \m\e^3\x\e^2\cdot\x=\m\e^3\x,~~ \\
    & \m\e^2\x\e\x\e^2\cdot\x=\m\e^2\x\e\x,~~ \m\e^2\x\e\x\e^2\cdot\e=\m\e^2\x\e\x,~~ \m\e^3\x\e\x\e\cdot\x=\m\e^3\x,~~ \m\e^3\x\e^3\cdot\x=\m\e\x^2,~~ \\
    & \m\e^3\x\e^3\cdot\e=\m\e^3\x\e,~~ \m\e^4\x\e\x\cdot\x=\m\x^2,~~ \m\e^4\x\e^2\cdot\x=\m\e^4\x,~~ \m\e^6\x\cdot\x=\m\x^2,~~ \\
    & \m\e^7\cdot\x=\m\e\x,~~ \m\e^3\x\e\x\e^2\cdot\x=\m\e^3\x\e\x,~~ \m\e^4\x\e\x\e\cdot\x=\m\e^4\x,~~ \m\e^4\x\e^3\cdot\x=\m\x^2,~~\\
    & \m\e^4\x\e^3\cdot\e=\m\e^4\x\e,~~ \m\e^8\cdot\x=\m\x^2,~~ \m\e^3\x\e\x\e^3\cdot\x=\m\x^2,~~ \m\e^3\x\e\x\e^3\cdot\e=\m\e^3\x\e\x\e,~~ \\
    & \m\e^4\x\e\x\e^2\cdot\x=\m\e^4\x\e\x,~~ \m\e^6\x\e\x\cdot\x=\m\e\x^2,~~ \m\e^6\x\e^2\cdot\x=\m\e^6\x,~~ \m\e^9\cdot\x=\m\e^3\x,~~ \\
    & \m\e^9\cdot\e=\m\e^7,~~ \m\e^4\x\e\x\e^3\cdot\x=\m\e\x^2,~~ \m\e^4\x\e\x\e^3\cdot\e=\m\e^4\x\e\x\e,~~ \m\e^6\x\e\x\e\cdot\x=\m\e^6\x,~~ \\
    & \m\e^6\x\e^3\cdot\x=\m\e\x^2,~~ \m\e^6\x\e^3\cdot\e=\m\e^6\x\e,~~ \m\e^6\x\e\x\e^2\cdot\x=\m\e^6\x\e\x,~~ \m\e^6\x\e\x\e^3\cdot\x=\m\x^2,~~ \\
    & \m\e^6\x\e\x\e^3\cdot\e=\m\e^6\x\e\x\e.
\end{align*}\normalsize

In fact, we can also give the SG-value DFA for this game; its structure is more complex, with 71 states and 141 transitions.
The classification of its states is as follows:
\tiny
\begin{align*}
    \varU^{\leq}: & \epsilon, \m, \m\x, \m\e; \\
    \varU^>_0:    & \m\x\e, \m\e\x, \m\e\x^2, \m\e^3, \m\x\e^3, \m\e\x\e^2, \m\e^2\x\e, \m\e^4\x, \m\e^3\x\e\x, \m\e^3\x\e^2, \\
                  & \m\e^6, \m\e^2\x\e\x\e^2, \m\e^3\x\e\x\e, \m\e^3\x\e^3, \m\e^7, \m\e^6\x\e, \m\e^6\x\e\x, \m\e^6\x\e^2, \m\e^9, \m\e^4\x\e\x\e^3, \\
                  & \m\e^6\x\e\x\e, \m\e^6\x\e^3, \m\e^{10}, \m\e^6\x\e\x\e^2, \m\e^{10}\x, \m\e^{12}, \m\e^{10}\x\e^2, \m\e^{13}, \m\e^{10}\x\e\x\e, \m\e^{10}\x\e\x\e^3; \\
    \varU^>_1:    & \m\x^2, \m\e^2, \m\x\e^2, \m\e\x\e, \m\x\e^2\x, \m\e^3\x, \m\e\x\e^3, \m\e^2\x\e^2, \m\e^3\x\e, \m\e^2\x\e\x\e, \\
                  & \m\e^4\x\e\x, \m\e^6\x, \m\e^3\x\e\x\e^2, \m\e^4\x\e^3, \m\e^8, \m\e^3\x\e\x\e^3, \m\e^8\x\e^2, \m\e^6\x\e\x\e^3, \m\e^8\x\e\x\e, \m\e^{10}\x\e\x, \\
                  & \m\e^{10}\x\e^3; \\
    \varU^>_2:    & \m\e^2\x, \m\e^2\x^2, \m\e^4, \m\e^2\x\e\x, \m\e^5, \m\e^4\x\e, \m\e^4\x\e^2, \m\e^4\x\e\x\e, \m\e^4\x\e\x\e^2, \m\e^8\x, \\
                  & \m\e^8\x\e, \m\e^8\x\e\x, \m\e^{11}, \m\e^{10}\x\e, \m\e^8\x\e\x\e^2, \m\e^{10}\x\e\x\e^2.
\end{align*}\normalsize
Hence in the generalized misère quotient $\varU/\cong^{\varG}$ for SG values of this game, there are respectively $30, 21, 16$ equivalence classes for SG values $0, 1, 2$.

\subsubsection{Outcome DFAs for Quasi-Reachable Positions}

If we restrict positions to quasi-reachable ones, then the outcome DFA must additionally distinguish quasi-reachability, which requires introducing new states;
however, since fewer positions need to be solved, the number of states also decreases accordingly, especially when $S$ is large.
In our computational results, the outcome DFAs restricted to quasi-reachable positions are generally much simpler than those for all positions, with the exception of $S=\{1, 2\}$, where the number of states actually increases.
Nevertheless, even when only determining outcomes of all quasi-reachable positions, our algorithm can only produce outcome DFAs for cases with $s_{\max}\leq 3$; it still fails for non-trivial games with larger $s_{\max}$.

The following table lists the complexity of outcome DFAs for non-trivial games successfully solved by our algorithm
\footnote{In the table, games are denoted by their move sets; those with a superscript $-$ are misère play, otherwise normal play; $|Q_F|$ lists the number of final states of each DFA in order.}:

\begin{center}
    \begin{tabular}{|c|c|c|c|}
        \hline
        Game & $|K|$ & $|\delta|$ & $|Q_F|$ \\
        \hline
        $\{1, 2\}$               & 91  & 139  & 15,25 \\
        $\{1, 2\}^-$             & 91  & 139  & 19,21 \\
        $\{1,3\}\col\{1,2\}$     & 351 & 602  & 58,97 \\
        $\{1,3\}\col\{1,2\}^-$   & 705 & 1226 & 156,113 \\
        $\{1,3\}$                & 248 & 427  & 48,66 \\
        $\{1,3\}^-$              & 426 & 739  & 90,113 \\
        $\{2,3\}\col\{1,2\}$     & 261 & 450  & 42,68 \\
        $\{2,3\}\col\{1,2\}^-$   & 575 & 976  & 107,160 \\
        $\{2,3\}$                & 280 & 469  & 61,69 \\
        $\{2,3\}^-$              & 349 & 591  & 73,92 \\
        $\{1,2,3\}\col\{1,2\}$   & 123 & 220  & 22,21 \\
        $\{1,2,3\}\col\{1,2\}^-$ & 531 & 924  & 88,157 \\
        $\{1,2,3\}\col\{1,3\}$   & 647 & 1118 & 150,157 \\
        $\{1,2,3\}\col\{2,3\}$   & 718 & 1258 & 158,184 \\
        \hline
    \end{tabular}
\end{center}

From these results, one can see that even when restricted to quasi-reachable positions, the number of states of the outcome DFA grows very rapidly with the move set (very likely still super-exponentially).
Moreover, during our solving process we did not impose the restriction that positions in the same equivalence class must share the same rule label ($\rmL$ or $\rmR$); unlike the case of partizan octal games in the previous chapter, it occurs frequently that positions with different labels are equivalent to each other.

The following table gives the outcomes of the empty board (starting from $n=1$) and their periods for all cases with $s_{\max}\leq 3$ (those marked with ? are conjectured periods):
\begin{center}
    \small
    \begin{tabular}{|c|c|c|c|}
        \hline
        Game & Outcomes & $(p, t)$ \\
        \hline
        $\{1,2\}$                 & ${}_{\tabP\tabN\tabP\tabN\tabN\tabN\tabP\tabN\tabN\tabN\tabN\tabN\tabN\tabN\tabN\tabN\ldots}$ & $(7,1)$ \\
        $\{1,2\}^-$               & ${}_{\tabN\tabP\tabN\tabP\tabN\tabN\tabP\tabP\tabN\tabP\tabP\tabN\tabP\tabP\tabN\tabP\ldots}$ & $(5,3)$ \\
        $\{1,3\}\col\{1,2\}$      & ${}_{\tabP\tabN\tabR\tabN\tabN\tabN\tabN\tabR\tabN\tabN\tabN\tabN\tabR\tabN\tabN\tabN\ldots}$ & $(1,5)$ \\
        $\{1,3\}\col\{1,2\}^-$    & ${}_{\tabN\tabP\tabN\tabP\tabR\tabP\tabR\tabR\tabR\tabR\tabR\tabR\tabR\tabR\tabR\tabR\ldots}$ & $(6,1)$ \\
        $\{1,3\}$                 & ${}_{\tabP\tabN\tabP\tabN\tabN\tabN\tabP\tabN\tabP\tabN\tabN\tabN\tabP\tabN\tabP\tabN\ldots}$ & $(0,6)$ \\
        $\{1,3\}^-$               & ${}_{\tabN\tabP\tabN\tabP\tabN\tabN\tabN\tabP\tabN\tabP\tabN\tabN\tabN\tabP\tabN\tabP\ldots}$ & $(0,6)$ \\
        $\{2,3\}\col\{1,2\}$      & ${}_{\tabP\tabN\tabR\tabN\tabL\tabN\tabR\tabN\tabL\tabN\tabR\tabN\tabL\tabN\tabR\tabN\ldots}$ & $(1,4)$ \\
        $\{2,3\}\col\{1,2\}^-$    & ${}_{\tabN\tabP\tabN\tabN\tabN\tabP\tabN\tabN\tabN\tabP\tabN\tabN\tabN\tabP\tabN\tabN\ldots}$ & $(0,4)$ \\
        $\{2,3\}\col\{1,3\}$      & ${}_{\tabP\tabN\tabN\tabN\tabP\tabN\tabN\tabN\tabN\tabN\tabR\tabN\tabN\tabN\tabN\tabN\ldots}$ & $(5,6)?$ \\
        $\{2,3\}\col\{1,3\}^-$    & ${}_{\tabN\tabP\tabP\tabP\tabN\tabP\tabN\tabP\tabP\tabP\tabP\tabP\tabN\tabP\tabP\tabP\ldots}$ & $(5,6)?$ \\
        $\{2,3\}$                 & ${}_{\tabP\tabN\tabP\tabN\tabP\tabN\tabP\tabP\tabN\tabN\tabP\tabP\tabP\tabN\tabN\tabN\ldots}$ & $(11,5)$ \\
        $\{2,3\}^-$               & ${}_{\tabN\tabP\tabN\tabP\tabN\tabN\tabN\tabN\tabN\tabN\tabN\tabN\tabN\tabN\tabN\tabN\ldots}$ & $(4,1)$ \\
        $\{1,2,3\}\col\{1,2\}$    & ${}_{\tabP\tabN\tabP\tabN\tabL\tabL\tabL\tabL\tabL\tabL\tabL\tabL\tabL\tabL\tabL\tabL\ldots}$ & $(4,1)$ \\
        $\{1,2,3\}\col\{1,2\}^-$  & ${}_{\tabN\tabP\tabN\tabP\tabN\tabP\tabN\tabL\tabN\tabL\tabL\tabL\tabL\tabL\tabL\tabL\ldots}$ & $(9,1)$ \\
        $\{1,2,3\}\col\{1,3\}$    & ${}_{\tabP\tabN\tabL\tabN\tabL\tabN\tabL\tabL\tabL\tabL\tabL\tabL\tabL\tabL\tabL\tabL\ldots}$ & $(7,1)$ \\
        $\{1,2,3\}\col\{1,3\}^-$  & ${}_{\tabN\tabP\tabN\tabP\tabN\tabL\tabN\tabL\tabL\tabL\tabL\tabL\tabL\tabL\tabL\tabL\ldots}$ & $(7,1)?$ \\
        $\{1,2,3\}\col\{2,3\}$    & ${}_{\tabP\tabN\tabL\tabN\tabP\tabN\tabL\tabL\tabL\tabL\tabL\tabL\tabL\tabL\tabL\tabL\ldots}$ & $(7,1)$ \\
        $\{1,2,3\}\col\{2,3\}^-$  & ${}_{\tabN\tabP\tabN\tabP\tabN\tabL\tabN\tabP\tabN\tabL\tabL\tabL\tabL\tabL\tabL\tabL\ldots}$ & $(9,1)?$ \\
        $\{1,2,3\}$               & ${}_{\tabP\tabN\tabP\tabN\tabP\tabN\tabN\tabN\tabN\tabN\tabP\tabN\tabN\tabN\tabP\tabN\ldots}$ & $(7,4)?$ \\
        $\{1,2,3\}^-$             & ${}_{\tabN\tabP\tabN\tabP\tabN\tabP\tabN\tabP\tabN\tabN\tabN\tabP\tabN\tabN\tabN\tabN\ldots}$ & $(12,1)?$ \\
   \hline
   \end{tabular}
    \normalsize
\end{center}

In fact, the case $S=\{1, 4\}$ has known results for all reachable positions~\cite{Kotzig2014}, but our algorithm failed to produce them.
This is likely because our quasi-reachable condition still includes too many unreachable positions (especially when the move set is more ``sparse''), leading to an increase in the number of states.
However, further tightening the restriction may cause the set of positions to fail to be closed under $\nxt(\cdot)$, or to fail to be a regular language, which would invalidate our verification algorithm.

\subsubsection{SG Values for Quasi-Reachable Positions}

If we restrict positions to quasi-reachable ones, our algorithm can also produce outcome DFAs for the cases $S=\{1, 3\}$ and $\{2, 3\}$. These are shown below:

\begin{center}
    \begin{tabular}{|c|c|c|c|}
        \hline
        Game & $|Q|$ & $|\delta|$ & $|Q_F|$\\
        \hline
        $\{1,2\}$ & 49  & 81  & 14,21,20 \\
        $\{1,3\}$ & 177 & 313 & 64,79,34 \\
        $\{2,3\}$ & 200 & 350 & 61,73,65 \\
        \hline
    \end{tabular}
\end{center}

\begin{center}
    \small
    \begin{tabular}{|c|c|c|c|}
        \hline
        Game & SG sequence & $(p, t)$ \\
        \hline
        $\{1,2\}$ & ${}_{0,1,0,1,2,2,0,1,2,1,2,2,1,1,2,1,1,2,1,1,2,1,1,2,1,1\ldots}$ & $(11,3)$ \\
        $\{1,3\}$ & ${}_{0,1,0,1,2,1,0,1,0,1,2,1,0,1,0,1,2,1,0,1,0,1,2,1,0,1\ldots}$ & $(0,6)$ \\
        $\{2,3\}$ & ${}_{0,1,0,1,0,1,0,0,2,2,0,0,0,2,2,2,0,0,2,2,2,0,0,2,2,2\ldots}$ & $(11,5)$ \\
        \hline
   \end{tabular}
    \normalsize
\end{center}

Finally, it is worth mentioning that for some move sets of larger size, the SG values appear to exhibit very simple periodicity:
when the board is sufficiently long, positions of even length have SG value 1, and otherwise 0. For example:
\begin{center}
    \begin{tabular}{|c|c|}
        \hline
        Game & SG sequence\\
        \hline
        $\{2,3,4,6\}$    & ${}_{0,1,0,1,0,1,0,1,0,1,0,1,0,1,0,1,0,1,0,1,0,1,0,1,0,1,\ldots}$ \\
        $\{1,2,3,4,6\}$  & ${}_{0,1,0,1,0,1,0,1,0,1,0,1,0,1,0,1,0,1,0,1,0,1,0,1,0,1,\ldots}$ \\
        $\{1,2,3,5,6\}$  & ${}_{0,1,0,1,0,1,0,1,0,1,0,1,0,1,0,1,0,1,0,1,0,1,0,1,0,1,\ldots}$ \\
        $\{1,2,3,4,7\}$  & ${}_{0,1,0,1,0,1,0,1,0,1,0,1,0,1,0,1,0,1,0,1,0,1,0,1,0,1,\ldots}$ \\
        \hline
    \end{tabular}
\end{center}
There is a simple but not rigorous explanation: if the available step sizes are sufficiently abundant that neither player can block the other's moves, then the token will visit every cell, and the parity of the number of cells determines which player makes the last move.
We also attempted to solve such cases, but all of them failed due to state-count overflow.
We suspect that Kotzig's nim exhibiting such periodicity are in fact still extremely complex, and that this apparent periodicity may be spurious: longer boards beyond the computable range may break the pattern.

\section{Open Problems}

In addition to the problems already listed above, in this chapter we present some further open problems that we consider suitable for future research directions.

\subsection{More Games}

The most obvious question is of course:
\begin{problem}
    Besides the games mentioned in this paper, what other games can be solved using finite-automaton methods?
\end{problem}
Our algorithm can be attempted on virtually any game whose positions are based on multi-heaps or one-dimensional boards, as long as its rules can be expressed using operations that are closed on regular languages; it is just not guaranteed to always succeed in producing an outcome DFA.

In the chapter on Kotzig's nim, we noted that if the SG values of a game are bounded, it may be possible to determine the SG values of positions using finitely many DFAs.
For games with unbounded SG values, if their SG values exhibit sufficiently strong regularities, solving them with finite automata is not hopeless; one only needs to find a way to compress the SG values according to their patterns.

For example, \textit{hexadecimal games}, which are octal games extended by allowing a heap to be split into three nonempty parts, include special cases whose SG values, though unbounded, are known to exhibit \textit{arithmetic periodicity},
meaning that each term in a period is always larger than the corresponding term in the previous period by a fixed value (called the \textit{amplitude})~\cite{2004Hexa}.
For instance, the SG sequence of \textbf{0.3F} is
\[
    0, 1, 2, 0, 1, 2, 3, 4, 5, 3, 4, 5, 6, 7, 8, 6, 7, 8, \ldots,
\]
with preperiod, period and amplitude $0, 6, 3$, respectively.

Besides hexadecimal games, Wythoff's game also exhibits arithmetic periodicity.
In proving the arithmetic periodicity of Wythoff's game,~\cite{2002Wyth} introduced the mapping
\[
    \varH(m, n) \coloneq \varG(m,n) - m + 2n
\]
to successfully compress the two-dimensional SG values $(m,n)$ into a bounded range, making them recognisable by finite automata.

Beyond arithmetic periodicity, hexadecimal games have also been found to exhibit other more complex patterns, such as so-called \textit{ruler regularity}~\cite{2020Hexa}.

The SG values of impartial hexadecimal games under normal play are given by the Sprague--Grundy theorem, but research on misère play and partizan versions remains a blank.
If some special cases of these variants still exhibit sufficiently strong regularities, similar compression of their SG values may be possible.

\begin{problem}
    Can this method be generalized to determine the outcomes or SG values of hexadecimal games exhibiting arithmetic periodicity (or other special periodicities)?
\end{problem}

\subsection{Stronger Automata}

The rules of many games cannot be expressed as operations on finite automata.
Other games, though expressible, have outcomes that cannot be determined by finite automata.
For example, the algebraically periodic octal games mentioned earlier exhibit clear regularities in their outcomes under misère play, but these go beyond the class of languages recognisable by finite automata.
This is because they require comparing the largest heap with the sum of the remaining heaps, but finite automata cannot perform counting (unless only finitely many numbers are involved), and hence cannot compare magnitudes.
With this in mind:

\begin{problem}
    Can our approach be generalized to stronger automata?
\end{problem}

The most common automata stronger than finite automata are \textit{deterministic pushdown automata} (DPDA), which have an unbounded last-in-first-out stack for recording information, and can recognise many non-regular languages, including palindromes, balanced parentheses, and many languages that depend on counting ability.
However, the properties of DPDAs are far less well-behaved than those of FAs: the languages they recognise (called \textit{context-free languages}) do not enjoy such good closure properties under various operations, and are sometimes even non-computable.
Their nondeterministic counterpart, NPDA, is even more powerful, but computability of various operations becomes correspondingly worse.

Given the trade-off between power and computability, automata intermediate between FAs and NPDAs seem more worthy of further exploration,
such as \textit{visibly pushdown automata} (VPDA)~\cite{VPDA2004} and \textit{one-counter automata} (OCA)~\cite{OCA2018}.

Our earlier argument that the outcomes of misère octal games cannot be determined by FAs was not very rigorous.
We do not intend to expand on this discussion here, but instead subsume it into the following open problem:

\begin{problem}
    For a given class of automata (especially FAs), is there a rigorous method to prove that the outcomes of a given game cannot be determined by automata in that class?
\end{problem}

This would help us better understand the computational complexity of games.
Moreover, if some octal game could be proven FA-insoluble, that would suffice to disprove Guy's conjecture, though this appears very difficult.

\subsection{Weak Solutions}

Sometimes a complete solution of a game is very complex, but we may settle for a \textit{weak solution}: a solution for a special subset of positions that includes the initial positions of interest.
The quasi-reachable positions defined in this paper are an instance of this approach.

Another example is \textit{Treblecross}. The rules are that two players alternately place their own tokens on a one-dimensional board, with the restriction that no three consecutive tokens may be formed; the player who cannot move loses.
\cite{2026TrebleCross} proved that the outcome of an empty board of length $n$ is determined by $n \bmod 10$: it is a $\varP$-position when the remainder is $0,2,3,6$ or $9$, and an $\varN$-position otherwise.
Their proof proceeded by giving the outcomes of a special class of positions called ``regular positions'', and showing that from an empty board, one player can always keep the position within the set of regular positions.

In fact, the set of all positions in the above class forms a regular language, and there exists an outcome DFA restricted to this class.
However, unlike quasi-reachable positions, this regular language is not closed under $\nxt(\cdot)$.
Moreover, our algorithm fails to produce an outcome DFA for all positions of this game, and from the growth trend of the state queue, it is likely that no such DFA exists at all.
Therefore we ask:

\begin{problem}
    Can the algorithm be improved so that it can automatically find and solve weak solutions in the above situation?
\end{problem}

\subsection{Two-Dimensional Generalizations}

Many impartial games are played on two-dimensional boards, including classical games such as Go, Gomoku, and Hex. Hence we ask:

\begin{problem}
    Does there exist a two-dimensional generalization of our approach?
\end{problem}

Such a generalization is likely to be quite difficult in practice.
Although there exist two-dimensional analogues of regular languages and finite automata, such as \textit{four-way automata} (4FA)~\cite{2dFA1967} and \textit{recognisable picture languages} (REC)~\cite{2dFA1997},
their properties are far worse than their one-dimensional counterparts, and many related problems become non-computable.
Nevertheless, for many open problems on two-dimensional board games, this seems a promising avenue,
especially for games like Gomoku, which have ``local'' move rules and winning conditions, and appear well-suited for representation by specific two-dimensional variants of finite automata.

\section*{Acknowledgments}
We are grateful to Alex Gu and Rafał Wrona for their careful reading of the preprint and for bringing an error in Corollary~\ref{sg_cond} to our attention.

\section*{Declaration of AI-assisted Technologies}

The original text of this manuscript was entirely composed and subsequently checked by the author(s).
Certain tables and figures were produced from the execution of Python scripts.
DeepSeek was employed exclusively for the purposes of refining language and rectifying typographical errors; it was not utilized for the generation of textual content.

\bibliographystyle{plain}
\bibliography{references}

@article{Octal1956,
  title = {The G-values of various games},
  author = { Guy, Richard K. and Smith, Cedric A. B. },
  journal = {Mathematical Proceedings of the Cambridge Philosophical Society},
  volume = {52},
  number = {3},
  pages = {514-526},
  year = {1956},
}

@article{Part1976,
  title = {Impartial and partisan games},
  author = { Austin, Richard B. },
  journal = {Mathematics and Statistics University of Calgary},
  year = {1976},
}

@article{PartOctal1987,
  author = {Fraenkel, A. S. and Kotzig, A.},
  title = {Partizan octal games: Partizan subtraction games},
  journal = {International Journal of Game Theory},
  year = {1987},
  volume = {16},
  number = {2},
  pages = {145--154},
  doi = {10.1007/BF01780638},
  mrnumber = {MR0899792},
  zbmath = {0662.90095},
  publisher = {Springer}
}

@article{MisQuot2006,
  title = {Misere quotients for impartial games},
  author = { Plambeck, Thane E. and Siegel, Aaron N. },
  year = {2006},
}

@misc{PartMisQuot2007,
  title = {Misère canonical forms of partizan games},
  author = { Siegel, Aaron N. },
  year = {2007},
}

@article{PartMisQuot2015,
  title = {Peeking at partizan misère quotients},
  author = { Allen, Meghan R. },
  journal = {Games of No Chance 4},
  volume = {63},
  year = {2015},
}

@article{PartMis2010,
  title = {An Investigation of Partizan Misere Games},
  author = { Allen, Meghan R. },
  journal = {Mathematics},
  year = {2010},
}

@article{PartMisQuot2013,
  author = {Milley, Rebecca and Renault, Gabriel},
  title = {Dead ends in misère play: The misère monoid of canonical numbers},
  journal = {Discrete Mathematics},
  year = {2013},
  volume = {313},
  number = {20},
  pages = {2223--2231},
  publisher = {Elsevier BV},
  doi = {10.1016/j.disc.2013.05.023},
  url = {http://dx.doi.org/10.1016/j.disc.2013.05.023},
  issn = {0012-365X}
}

@inproceedings{PartMisQuot2023,
  title = {On the General Dead-Ending Universe of Partizan Games},
  author = {Aaron N. Siegel},
  year = {2023},
  url = {https://api.semanticscholar.org/CorpusID:266573600}
}

@article{VertexDel2013,
  title = {Vertex Deletion games with Parity rules},
  author = { Nowakowski, Richard J and Ottaway, Paul },
  journal = {Integers},
  volume = {5},
  number = {2},
  year = {2013},
}

@article{Kotzig1995,
  author    = {Fraenkel, Aviezri S. and Jaffray, Alan and Kotzig, Anton and Sabidussi, Gert},
  title     = {Modular Nim},
  journal   = {Theoretical Computer Science},
  year      = {1995},
  volume    = {143},
  number    = {2},
  pages     = {319--333},
  doi       = {10.1016/0304-3975(94)00260-p},
  publisher = {Elsevier BV}
}

@article{Kotzig2005,
  title={Winning Ways, for Your Mathematical Plays},
  author={ Berlekamp, Elwyn R  and  Conway, John H  and  Guy, Richard K },
  journal={Math Horizons},
  volume={13},
  number={2},
  pages={28-29},
  year={2005},
}

@article{Kotzig2014,
  author    = {Tan, Xin Lu and Ward, Mark Daniel},
  title     = {On Kotzig's Nim},
  journal   = {Integers},
  year      = {2014},
  volume    = {14},
  pages     = {Paper G06, 27 p.},
  mrnumber  = {MR3278116},
  zbmath    = {1314.91056}
}

@article{Kotzig2019,
  title={Sprague-Grundy Values of Modular Nim},
  author={ Horrocks, D. G.  and  Horrocks, Jonathan H. },
  journal={Integers},
  volume={19},
  pages={G1},
  year={2019},
}

@article{Kotzig2025,
  author    = {Arun, Srinivas},
  title     = {Geography, Kotzig's Nim, and variants},
  journal   = {Theoretical Computer Science},
  year      = {2025},
  volume    = {1023},
  pages     = {Article No. 114957, 14 p.},
  doi       = {10.1016/j.tcs.2024.114957},
  publisher = {Elsevier}
}

@article{GONC,
  title={Unsolved Problems in Combinatorial Games},
  author={ Guy, Richard K. },
  journal={R.j.nowakowski Ed.games of},
  volume={329},
  pages={475--491},
  year={1995},
}

@article{subs_game,
  title={A brief conversation about subtraction games},
  author={ Larsson, Urban  and  Saha, Indrajit },
  year={2024},
}

@article{2025Wyth,
  title={Variants of Wythoff game with terminal positions or blocking maneuvers},
  author={ Renard, Antoine  and  Rigo, Michel },
  year={2025},
}

@article{2025Auto,
  title={Automatic proofs in combinatorial game theory},
  author={ Mignoty, Bastien  and  Renard, Antoine  and  Rigo, Michel  and  Whiteland, Markus A. },
  journal={International Journal of Game Theory},
  volume={54},
  number={2},
  pages={1-32},
  year={2025},
}

@article{1984Algper,
  author = {Allemang, D., T.},
  title = {Machine computation with finite games},
  journal = {Master’s thesis, Trinity College, Cambridge},
  year = {1984},
}

@article{2010Wyth,
  title={Extensions and restrictions of Wythoff's game preserving its P positions},
  author={ric Duchêne and  Fraenkel, Aviezri S.  and  Nowakowski, Richard J.  and  Rigo, Michel },
  journal={Journal of Combinatorial Theory Series A},
  volume={117},
  number={5},
  pages={545-567},
  year={2010},
}

@article{1982Wyth,
  title={How to Beat Your Wythoff Games' Opponent on Three Fronts},
  author={ Fraenkel, Aviezri S. },
  journal={The American Mathematical Monthly},
  volume={89},
  number={6},
  pages={353-361},
  year={1982},
}

@article{2002Wyth,
  title={A simple FSM-based proof of the additive periodicity of the Sprague-Grundy function of Wythoff's game},
  author={ Landman, Howard A. },
  year={2002},
}

@misc{2026TrebleCross,
      title={A Weak Solution of Inverse Treblecross},
      author={Kai Liang and Muxi Li},
      year={2026},
      eprint={2604.16759},
      archivePrefix={arXiv},
      primaryClass={math.CO},
      url={https://arxiv.org/abs/2604.16759},
}

@article{2004Hexa,
title = {Periodicity and arithmetic-periodicity in hexadecimal games},
journal = {Theoretical Computer Science},
volume = {313},
number = {3},
pages = {463-472},
year = {2004},
issn = {0304-3975},
doi = {https://doi.org/10.1016/j.tcs.2003.08.013},
url = {https://www.sciencedirect.com/science/article/pii/S0304397503005942},
author = {S. Howse and R.J. Nowakowski},
}

@article{2020Hexa,
 author = {Dailly, Antoine and Duch{\^e}ne, {\'E}ric and Larsson, Urban and Paris, Gabrielle},
 title = {Partition games},
 fjournal = {Discrete Applied Mathematics},
 journal = {Discrete Appl. Math.},
 issn = {0166-218X},
 volume = {285},
 pages = {509--525},
 year = {2020},
 language = {English},
 doi = {10.1016/j.dam.2020.05.032},
 zbMATH = {7242117},
 Zbl = {1452.91056}
}

@inproceedings{VPDA2004,
  author    = {Rajeev Alur and P. Madhusudan},
  title     = {Visibly pushdown languages},
  booktitle = {Proceedings of the 36th Annual {ACM} Symposium on Theory of Computing},
  pages     = {202--211},
  year      = {2004},
  doi       = {10.1145/1007352.1007390},
}

@article{OCA2018,
  author    = {Vojtěch Forejt and Petr Jančar and Stefan Kiefer and James Worrell},
  title     = {Game characterization of probabilistic bisimilarity, and applications to pushdown automata},
  journal   = {Logical Methods in Computer Science},
  volume    = {14},
  number    = {4},
  year      = {2018},
}

@incollection{2dFA1967,
  author    = {M. Blum and C. Hewitt},
  title     = {Automata on a 2-dimensional tape},
  booktitle = {IEEE Symposium on Switching and Automata Theory},
  year      = {1967},
  pages     = {155-160},
}

@incollection{2dFA1997,
  author    = {Dora Giammarresi and Antonio Restivo},
  title     = {Two-dimensional languages},
  booktitle = {Handbook of Formal Languages},
  volume    = {3},
  pages     = {215-267},
  year      = {1997},
  publisher = {Springer},
}

\end{document}